\documentclass[11pt]{amsart}
\usepackage{graphicx} 
\usepackage{multicol}
\usepackage{makecell}
    
    \usepackage{chngpage} 
\usepackage{tabularx}
\usepackage{xcolor,amsmath,amssymb,amsthm,mathtools,graphicx}
\usepackage{tikz-cd}
\usepackage{hyperref}
\newtheorem{theorem}{Theorem}[section]
\newtheorem{lemma}[theorem]{Lemma}
\newtheorem{definition}[theorem]{Definition}
\newtheorem{example}[theorem]{Example}
\newtheorem{proposition}[theorem]{Proposition}

\newtheorem{corollary}[theorem]{Corollary}
\newtheorem{conjecture}[theorem]{Conjecture}

\newtheorem{List}[theorem]{List}

\theoremstyle{definition}
\newtheorem{remark}[theorem]{Remark}

\newcommand{\Z}{\mathbb{Z}}
\newcommand{\blank}{\rule{.3cm}{0.15mm}}
\newcommand{\dslash}{/\mkern-5mu/}
\newcommand{\dbslash}{\backslash\mkern-5mu\backslash}
\DeclareMathOperator{\Hom}{Hom}
\newcommand{\Mlt}{\mathop{\mathrm{Mlt}}}
\newcommand{\Sym}{\mathop{\mathrm{Sym}}}

\begin{document}


\title[Homology of quasigroups of Bol-Moufang type]{In search of homology for quasigroups of Bol-Moufang type}
\author[Christiana]{Anthony Christiana$^1$}
\address{$^1$Department of Mathematics, The George Washington University, \newline \indent Washington DC, USA}
\email{{\rm $^1$ajchristiana@gwmail.gwu.edu}}

\author[Clingenpeel]{Ben Clingenpeel$^2$}
\address{$^2$Department of Mathematics, The George Washington University,\newline \indent Washington DC, USA}
\email{{\rm $^2$ben.clingenpeel@gwu.edu}}

\author[Guo]{Huizheng Guo$^3$} 
\address{$^3$Department of Mathematics, The George Washington University, \newline \indent Washington DC, USA}
\email{{\rm $^3$hguo30@gwu.edu}}
\author[Oh]{Jinseok Oh$^4$}
\address{$^4$Department of Mathematics, Kyungpook National University, Daegu, \newline \indent South Korea}
\email{{\rm $^4$jinseokoh@knu.ac.kr}}
\author[Przytycki]{J\'{o}zef H. Przytycki$^5$}
\address{$^5$Department of Mathematics, The George Washington University,\newline \indent Washington DC, USA and \newline \indent Department of Mathematics, University of Gda\'{n}sk, Gda\'{n}sk, Poland}
\email{{\rm $^5$przytyck@gwu.edu}}

\author[Zamojska-Dzienio]{Anna Zamojska-Dzienio$^6$}
\address{$^6$Faculty of Mathematics and Information Science\\ Warsaw
University of\newline \indent Technology, 00-661 Warsaw, Poland}
\email{$^6$anna.zamojska@pw.edu.pl}
\begin{abstract}
    We initiate (co)homology theory for quasigroups of Bol-Moufang type based on analysis of their extensions by affine quasigroups of the same type. We use these extensions to define second and third boundary operations, $\partial_2(x,y)$ and $\partial_3(x,y,z)$, respectively. We use these definitions to compute the second homology groups for several examples from the work of Phillips and Vojtechovsky. We speculate about the relation between these homology groups and those obtained from a small category with coefficients in a functor.
\end{abstract}
\maketitle
\tableofcontents

\section{Introduction}

The goal of our paper is to  initiate a (co)homology theory for quasigroups of Bol-Moufang (BMq) type \cite{Bo,Mo,PhVo1}. Our approach, which has its roots in the work of Eilenberg and his coauthors \cite{Eil,EM,CE}, is to analyze extensions of a quasigroup $(X,*_X)$ by an affine quasigroup $(A,*_A)$ of the same type. We study these extensions not to classify them, but to have a first glimpse at the homology of BMq quasigroups via their second and third boundary operations, $\partial_2(x,y)$ and $\partial_3(x,y,z)$ respectively. We compute the second homology groups for all \emph{distinguishing examples} of BMq described by Phillips and Vojtechovsky \cite{PhVo1}.

In Section~\ref{S:QBM} we recall the definition of quasigroups and provide the foundation for a discussion of the classification results of Phillips and Vojtechovsky \cite{PhVo1}. We then define one-sided loops and affine quasigroups.
\color{black}

In Section \ref{S:EH} we take cues from \cite{Eil} in constructing an extension of a quasigroup $(X,*_X)$ by an affine quasigroup $(A,*_A)$  over an abelian group $(A,+)$ (that is $a*_Ab=ta+sb+c_0$) and derive boundary maps. We then discuss rooted binary trees related to the bracketing of algebraic expressions (words). We also list, in Subsection  \ref{sec:substitutions}, affine BMqs for each of the 26 BMq varieties described in \cite{PhVo1} and notice that in all of them $t$ and $s$ commute, assuming there are no zero divisors. Furthermore, $t=s=1$ is always one of the solutions (see List \ref{list}). In so doing we confirm for example that left loops necessarily have $s=1$ and right loops $t=1$.

In Section \ref{S:chain}, we verify that our boundary maps $\partial_2, \partial_3$ lead to a chain complex. Furthermore, we discuss the relation between the first homology and the `abelianization' of a BMq. This generalizes the fact that the first homology of a group is its abelianization.

In Section \ref{S:examples}, we compute the first and second homology of the BMq's which \cite{PhVo1, PhVo2} used to distinguish the 26 distinct varieties. We observe in our examples that the second homology of a BMq depends only on the variety of quasigroups and not on the choice of defining relation.

In Section \ref{S:summary}, we summarize our results and propose directions for future research. In particular, we speculate about building higher degree (co)homology for BMq, with one approach being to use the homology theory of small categories with coefficients in a functor to $k$-modules.

\section{Quasigroups of Bol-Moufang Type}\label{S:QBM} 

A binary algebra $(Q,\cdot)$ is a \emph{quasigroup} if given any two of $a,b,c$ as elements in $Q$ the third can be uniquely selected in $Q$ so that $a\cdot b=c$. A \emph{loop} is a quasigroup with a neutral element, so it can be understood as a \emph{not necessarily associative} group. However, it is worth emphasizing that quasigroup theory is not just a generalization of group theory but a discipline of its own with almost one hundred years of history as summarized in e.g. \cite{Pfl1}. This theory formally begins with two papers: \emph{Zur Struktur von Alternativkoerpern} by Ruth Moufang (1935) \cite{Mo}, and \emph{Gewebe und Gruppen} by Gerrit Bol (1937) \cite{Bo}, where they introduced algebraic structures known now as \emph{Moufang loops} and \emph{Bol loops}. For details and comprehensive references on quasigroup theory, see \cite{Pfl2}, \cite{Smi2} or \cite{Shch}.

Let $(M,\cdot)$ be a binary algebra and let $m$ be any fixed element in $M$. One considers \emph{translation maps} $L(m)\colon M\to M$ and $R(m)\colon M\to M$ defined by
\begin{equation*}
L(m)(x)=m\cdot x \;\;\text{ and }  \;\;  R(m)(x)=x\cdot m.
\end{equation*}
These maps can be now used to obtain the equivalent definition of a quasigroup as a binary algebra $(Q,\cdot)$ in which all translations are bijections. In particular, one finds that a quasigroup satisfies the \emph{cancellation laws}, that is, $x=y$ whenever $a,x,y\in Q$ with $a\cdot x=a\cdot y$ or $x\cdot a=y\cdot a$.

The properties of translations in a quasigroup $(Q,\cdot)$ allows us to study its \emph{(combinatorial) multiplication group}, denoted by $\Mlt{Q}$. This is a subgroup of a symmetry group $\Sym{Q}$ of a set $Q$ generated by 
\begin{equation*}
 \{L(q),R(q)\colon q\in Q\}. 
\end{equation*}


However, none of the definitions of a quasigroup presented so far can be written in terms of \emph{identities} (universally quantified formulas). The following description is necessary if one wants to work with equational reasoning software, as in \cite{PhVo1}. This is the reason why one introduces the third type of definition, this time as a universal algebra with three binary operations: multiplication $\cdot$, left division $\backslash$ and right division $/$. And since in this paper we rely heavily on results obtained in \cite{PhVo1}, we follow their \emph{universal algebraic} approach.

\subsection{Varieties of Quasigroups}
\subsubsection{Universal algebraic definition of a quasigroup}
\begin{definition}\label{Def:Qugp}
Let $Q$ be a set, and consider three binary operations on $Q$: \textit{multiplication} denoted by $\cdot$, $/$ for \textit{right division}, and $\backslash$ for \textit{left division}.  Consider also the following identities in these operations:
\begin{align*}
\text{(IL)}\phantom{A} y\backslash(y\cdot x)=x, && \text{(IR)}\phantom{A}x= (x\cdot y)/ y,\\
\text{(SL)}\phantom{A} y\cdot(y\backslash x)=x, && \text{(SR)}\phantom{A}x= (x/ y)\cdot y.
\end{align*}

\noindent If (IL) and (SL) hold, then $(Q, \cdot, \backslash)$ is a \emph{left quasigroup}.  If (IR) and (SR) hold, then $(Q, \cdot, /)$ is a \emph{right quasigroup}.  If all four identities hold, then $(Q, \cdot, \backslash, /)$ is a \emph{quasigroup}.\footnote{In the theory of racks and quandles we work with a right quasigroup and the operation $/$ is often denoted by $\bar \cdot$. 
}
\end{definition}
The \emph{variety of quasigroups} consists of all (universal) algebras with three binary operations that satisfy the identities (IL), (SL), (IR), (SR). According to Birkhoff's HSP theorem, a variety of quasigroups is then closed under homomorphic images, subalgebras, and (direct) products. 
For every set $X$, the variety contains a free algebra on $X$, and every algebra in a variety is a homomorphic image of a free algebra. For more information on universal algebra results and tools one should consult \cite{Berg} or \cite{Ber}.

Note that $y\backslash x$ is the unique solution to the equation $y\cdot z=x$, and similarly $x/y$ is the unique solution to the equation $z\cdot y=x$. In fact, the identities (IL), (SL) impose that each left translation is a bijection, and (IR), (SR) work similarly for right translations. We have then $L(y)^{-1}(x)=y\backslash x$ and $R(y)^{-1}(x)=x/y$. In particular, all three definitions of a quasigroup we presented here are equivalent. 

\subsubsection{Bol-Moufang identities}\label{Ss:BMid} 
Here we recall notation introduced in \cite{PhVo1} to keep our paper as self-contained as possible, but for all undefined notions or details one should consult the original reference.

Bol-Moufang Quasigroups are quasigroups defined equationally by so-called Bol-Moufang Identities. These identities equate a four element word in a quasigroup to itself under two different bracketings.

In what follows, we consider only words of length four created with respect to the operation $\cdot$ of multiplication and proper bracketing. Let $X=\{x_1<x_2<x_3<x_4\}$ be a set of ordered letters (variables). There are exactly 5 ways in which a word of length 4 can be bracketed:
\begin{table}    
\begin{center}
\begin{tabular}{c|c}
$1$ & $o(o(oo))$\\
$2$ & $o((oo)o)$\\
$3$ & $(oo)(oo)$\\
$4$ & $(o(oo))o$\\
$5$ & $((oo)o)o$
\end{tabular}
\end{center}
\caption{Ways of bracketing}\label{T:br}
\end{table}
One can now consider rooted binary trees with $4$ leaves enumerated with $x_1,x_2,x_3,x_4$ corresponding to these bracketings and obtain the following 5 words (see Figures \ref{F:bracket_1}-\ref{F:bracket_5}). Note we omit $\cdot$ while multiplying elements.

Moufang quasigroups satisfy an additional identity (within the variety of all quasigroups): 
\begin{equation*}
(xy)(zx)=(x(yz))x.
\end{equation*}
They form a subvariety of varieties of left Bol quasigroups: $x(y(xz))=(x(yx))z$ and of right Bol quasigroups: $x((yz)y)=((xy)z)y$. As we could see all these identities share common features:
\begin{enumerate}
\item the only operation involved is $\cdot$,
\item the same 3 variables appear on both sides, in the same order,
\item one of the variables appears twice on both sides,
\item the remaining two variables appear once on both sides.
\end{enumerate}
Identities $\alpha=\beta$ satisfying conditions 1.-4. mentioned above will be called of \emph{Bol-Moufang type}. A variety of quasigroups is of \emph{Bol-Moufang type} if it is defined by one additional identity of Bol-Moufang type (within the variety of all quasigroups).

Let $x,y,z$ be the variables appearing in an identity of Bol-Moufang type and assume that they appear in alphabetical order. There are exactly 6 ways in which the 3 variables can form a word of length 4:

\begin{table}[hbt]    
\begin{center}
\begin{tabular}{c|c}
    & $1\,2\,3\,4$ \\
    \hline
    $A$ & $x\,x\,y\,z$ \\
    $B$ & $x\,y\,x\,z$ \\
    $C$ & $x\,y\,y\,z$ \\
    $D$ & $x\,y\,z\,x$ \\
    $E$ & $x\,y\,z\,y$ \\
    $F$ & $x\,y\,z\,z$ \\

\end{tabular}
\end{center}
\caption{ Ordering of variables in identities of Bol-Moufang type}\label{T:id}
\end{table}

\setlength{\unitlength}{1mm}
\begin{figure}[hbt]
\begin{minipage}[b]{0.4\textwidth}
\begin{picture}(50,34)(-10,0)
\put(2,20){\makebox(0,0){$x_1$}} \put(4,22){\circle*{1}}
\put(11,11){\makebox(0,0){$x_2$}} \put(13,13){\circle*{1}}
\put(20,2){\makebox(0,0){$x_3$}} \put(22,4){\circle*{1}}
\put(42,2){\makebox(0,0){$x_4$}} \put(40,4){\circle*{1}}
\put(35,16){\makebox(0,0){$x_3x_4$}} \put(31,13){\circle*{1}}
\put(29,25){\makebox(0,0){$x_2(x_3x_4)$}} \put(22,22){\circle*{1}}
\put(11,34){\makebox(0,0){$x_1(x_2(x_3x_4))$}} \put(13,31){\circle*{1}}
\put(4,22){\line(1,1){9}} \put(7,10){\makebox(0,0){$ $}}
\put(13,13){\line(1,1){9}} \put(25,28){\makebox(0,0){$ $}}
\put(22,4){\line(1,1){9}} \put(37,28){\makebox(0,0){$ $}}
\put(40,4){\line(-1,1){9}} \put(19,10){\makebox(0,0){$ $}}
\put(31,13){\line(-1,1){9}} \put(16,19){\makebox(0,0){$ $}}
\put(22,22){\line(-1,1){9}} \put(28,19){\makebox(0,0){$ $}}
\end{picture}
\caption{Word $x_1(x_2(x_3x_4))$.}\label{F:bracket_1}
\end{minipage}
\hfill
\begin{minipage}[b]{0.4\textwidth}
\begin{picture}(40,34)(0,0)
\put(2,2){\makebox(0,0){$x_2$}} \put(4,4){\circle*{1}}
\put(24,2){\makebox(0,0){$x_3$}} \put(22,4){\circle*{1}}
\put(8,14){\makebox(0,0){$x_2x_3$}} \put(13,13){\circle*{1}}
\put(33,11){\makebox(0,0){$x_4$}} \put(31,13){\circle*{1}}
\put(28,25){\makebox(0,0){$(x_2x_3)x_4$}} \put(22,22){\circle*{1}}
\put(2,20){\makebox(0,0){$x_1$}} \put(4,22){\circle*{1}}
\put(12,34){\makebox(0,0){$x_1((x_2x_3)x_4)$}} \put(13,31){\circle*{1}}
\put(4,4){\line(1,1){9}} \put(7,10){\makebox(0,0){$ $}}
\put(4,22){\line(1,1){9}} \put(7,28){\makebox(0,0){$ $}}
\put(22,4){\line(-1,1){9}} \put(19,10){\makebox(0,0){$ $}}
\put(13,31){\line(1,-1){9}} \put(28,19){\makebox(0,0){$ $}}
\put(13,13){\line(1,1){9}} \put(16,19){\makebox(0,0){$ $}}
\put(22,22){\line(1,-1){9}} \put(19,28){\makebox(0,0){$ $}}
\end{picture}
\caption{Word $x_1((x_2x_3)x_4)$.}\label{F:bracket_2}
\end{minipage}
\end{figure}
\setlength{\unitlength}{1mm}
\begin{figure}[hbt]
    \centering
 \begin{picture}(70,30)(0,0)
\put(2,2){\makebox(0,0){$x_1$}} \put(4,4){\circle*{1}}
\put(26,2){\makebox(0,0){$x_2$}} \put(24,4){\circle*{1}}
\put(9,15){\makebox(0,0){$x_1x_2$}} \put(14,14){\circle*{1}}
\put(32,2){\makebox(0,0){$x_3$}} \put(34,4){\circle*{1}}
\put(56,2){\makebox(0,0){$x_4$}} \put(54,4){\circle*{1}}
\put(48,15){\makebox(0,0){$x_3x_4$}} \put(44,14){\circle*{1}}
\put(29,27){\makebox(0,0){$(x_1x_2)(x_3x_4)$}} \put(29,24){\circle*{1}}
\put(4,4){\line(1,1){10}} \put(7,10){\makebox(0,0){$ $}}
\put(14,14){\line(3,2){15}} \put(7,28){\makebox(0,0){$ $}}
\put(24,4){\line(-1,1){10}} \put(19,10){\makebox(0,0){$ $}}
\put(34,4){\line(1,1){10}} \put(28,19){\makebox(0,0){$ $}}
\put(54,4){\line(-1,1){10}} \put(16,19){\makebox(0,0){$ $}}
\put(44,14){\line(-3,2){15}} \put(19,28){\makebox(0,0){$ $}}
\end{picture}   
\caption{Word $(x_1x_2)(x_3x_4)$.}
\label{F:bracket_3}
\end{figure}
\setlength{\unitlength}{1mm}
\begin{figure}[hbt]
\begin{minipage}[b]{0.4\textwidth}
\begin{picture}(40,34)(-10,0)
\put(2,11){\makebox(0,0){$x_1$}} \put(4,13){\circle*{1}}
\put(6,24){\makebox(0,0){$x_1(x_2x_3)$}} \put(13,22){\circle*{1}}
\put(11,2){\makebox(0,0){$x_2$}} \put(13,4){\circle*{1}}
\put(26,15){\makebox(0,0){$x_2x_3$}} \put(22,13){\circle*{1}}
\put(33,2){\makebox(0,0){$x_3$}} \put(31,4){\circle*{1}}
\put(33,20){\makebox(0,0){$x_4$}} \put(31,22){\circle*{1}}
\put(21,34){\makebox(0,0){$(x_1(x_2x_3))x_4$}} \put(22,31){\circle*{1}}
\put(4,13){\line(1,1){9}} \put(7,10){\makebox(0,0){$ $}}
\put(13,22){\line(1,1){9}} \put(7,28){\makebox(0,0){$ $}}
\put(31,4){\line(-1,1){9}} \put(19,10){\makebox(0,0){$ $}}
\put(22,13){\line(-1,1){9}} \put(28,19){\makebox(0,0){$ $}}
\put(13,4){\line(1,1){9}} \put(16,19){\makebox(0,0){$ $}}
\put(31,22){\line(-1,1){9}} \put(19,28){\makebox(0,0){$ $}}
\end{picture}
\caption{Word $(x_1(x_2x_3))x_4$.}
\label{F:bracket_4}
\end{minipage}
\hfill
\begin{minipage}[b]{0.4\textwidth}
\begin{picture}(50,34)(0,0)
\put(2,2){\makebox(0,0){$x_1$}} \put(4,4){\circle*{1}}
\put(24,2){\makebox(0,0){$x_2$}} \put(22,4){\circle*{1}}
\put(8,15){\makebox(0,0){$x_1x_2 $}} \put(13,13){\circle*{1}}
\put(33,11){\makebox(0,0){$x_3$}} \put(31,13){\circle*{1}}
\put(16,25){\makebox(0,0){$(x_1x_2)x_3 $}} \put(22,22){\circle*{1}}
\put(42,20){\makebox(0,0){$x_4$}} \put(40,22){\circle*{1}}
\put(29,34){\makebox(0,0){$((x_1x_2)x_3)x_4 $}} \put(31,31){\circle*{1}}
\put(4,4){\line(1,1){9}} \put(7,10){\makebox(0,0){$ $}}
\put(22,22){\line(1,1){9}} \put(25,28){\makebox(0,0){$ $}}
\put(31,31){\line(1,-1){9}} \put(37,28){\makebox(0,0){$ $}}
\put(22,4){\line(-1,1){9}} \put(19,10){\makebox(0,0){$ $}}
\put(13,13){\line(1,1){9}} \put(16,19){\makebox(0,0){$ $}}
\put(22,22){\line(1,-1){9}} \put(28,19){\makebox(0,0){$ $}}
\end{picture}
\caption{Word $((x_1x_2)x_3)x_4$.}\label{F:bracket_5}
\end{minipage}

\end{figure}
    
    \subsubsection{Enumeration and classification due to \cite{PhVo1}}\label{Ss:dual}
In \cite{PhVo1} the authors fully classified varieties of quasigroups of Bol-Moufang type. They showed that although there are 60 ``different" identities of Bol-Moufang type, the exact number of such varieties is 26. This means that some of these identities define the same variety - one says then that identities are \emph{equivalent}. Moreover, in \cite{PhVo1} all inclusions between these 26 varieties were determined, and all necessary counterexamples - \emph{distinguishing examples} were provided (we return to these counterexamples in Section \ref{S:examples}). We use here notation from \cite{PhVo1} for naming the identities and varieties.

By $Vi$ with $V\in\{A,\ldots F\}$ and $i\in\{1,\ldots 5\}$ we denote the word whose variables are ordered according to $V$ (see Table \ref{T:id}) and bracketed according to $i$ (see Table \ref{T:br}). For example, $A2$ is the word $x((xy)z)$ which can be presented as the following binary tree (see Figure \ref{F:A2}).
\setlength{\unitlength}{1mm}
\begin{figure}[hbt]
    \centering
\begin{picture}(40,34)(0,0)
\put(2,2){\makebox(0,0){$x$}} \put(4,4){\circle*{1}}
\put(24,2){\makebox(0,0){$y$}} \put(22,4){\circle*{1}}
\put(8,14){\makebox(0,0){$xy$}} \put(13,13){\circle*{1}}
\put(33,11){\makebox(0,0){$z$}} \put(31,13){\circle*{1}}
\put(28,25){\makebox(0,0){$(xy)z$}} \put(22,22){\circle*{1}}
\put(2,20){\makebox(0,0){$x$}} \put(4,22){\circle*{1}}
\put(12,34){\makebox(0,0){$x((xy)z)$}} \put(13,31){\circle*{1}}
\put(4,4){\line(1,1){9}} \put(7,10){\makebox(0,0){$ $}}
\put(4,22){\line(1,1){9}} \put(7,28){\makebox(0,0){$ $}}
\put(22,4){\line(-1,1){9}} \put(19,10){\makebox(0,0){$ $}}
\put(13,31){\line(1,-1){9}} \put(28,19){\makebox(0,0){$ $}}
\put(13,13){\line(1,1){9}} \put(16,19){\makebox(0,0){$ $}}
\put(22,22){\line(1,-1){9}} \put(19,28){\makebox(0,0){$ $}}
\end{picture}
\caption{Word $A2:\;x((xy)z)$.}\label{F:A2}
\end{figure}
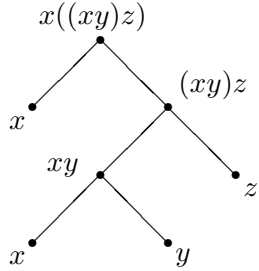

By $Vij$, with $V\in\{A,\ldots ,F\}$ and $1\leq i<j\leq 5$, we denote the identity whose left-hand side is $Vi$ and right hand-side is $Vj$. For example, the identity $A25$ corresponds to $A2=A5$, i.e. $x((xy)z)=((xx)y)z$.

The \emph{dual} to the identity $Vij$ is the identity $V'j'i'$ which can be calculated due to the following rules:
\begin{equation*}
A'=F,\;\;\;B'=E,\;\;\;C'=C,\;\;\;D'=D,\;\;\;1'=5,\;\;\;2'=4,\;\;\;3'=3,
\end{equation*}
with $\gamma''=\gamma \text{ for } \gamma=V \text{ or } \gamma=i$.

Defining identities for varieties of Bol-Moufang type are chosen in such a way that they are either self-dual or coupled into dual pairs (\cite[Table 1]{PhVo1}).

\subsection{(One-sided) Loops} 

In \cite{Kun} it was shown that a Moufang quasigroup $(Q,\cdot,\backslash,/)$, i.e. a quasigroup belonging to the variety defined by the one of equivalent identities: $B15$, $D23$, $D34$ or $E15$, is necesarrily a (Moufang) loop. This means there is a \emph{neutral element} $e\in Q$ such that $a\cdot e=e\cdot a=e\backslash a=a/e=a$ for every $a\in Q$.

Let $\mathcal{A}$ denote the abbreviation of the name of the Bol-Moufang type variety of quasigroups. In \cite{PhVo1} the information whether the variety $\mathcal{A}$ is a variety of (one-sided) loops is provided by the superscript following $\mathcal{A}$:
\begin{enumerate}
\item $\mathcal{A}^2$ means that every quasigroup in $\mathcal{A}$ is a loop;
\item $\mathcal{A}^L$ means that every quasigroup in $\mathcal{A}$ is a left loop, and there is a quasigroup in $\mathcal{A}$ that is not a right loop;
\item $\mathcal{A}^R$ is defined dually to $\mathcal{A}^L$, i.e. every quasigroup in $\mathcal{A}$ is a right loop, and there is a quasigroup in $\mathcal{A}$ that is not a left loop;
\item $\mathcal{A}^0$ means that there is a quasigroup in $\mathcal{A}$ that is not a left loop, and there is a quasigroup in $\mathcal{A}$ that is not a right loop.
\end{enumerate}
A left loop $(Q,\cdot,\backslash,/)$ possesses a left neutral element: $e\backslash a=e\cdot a=a$. Dually, a right loop $(Q,\cdot,\backslash,/)$ possesses a right neutral element: $a/e=a\cdot e=a$.

\subsection{Affine quasigroups}\label{Ss:AMBMQ} 
In the next sections, we will consider extensions of Bol-Moufang quasigroups by affine quasigroups of the same type. For this purpose we have to know the structure of affine BMq's. 
\begin{definition}
Let $(A,+)$ be an abelian group, $t,s\colon A\rightarrow A$ (not necessarily commuting) automorphisms of $(A,+)$ and $c_0\in A$. We define a new operation on $A$ by
\begin{equation}
a\ast b=t(a)+s(b)+c_0.    
\end{equation}
The resulting quasigroup $(A,\ast)$ is said to be \emph{affine} over $(A,+)$.
\end{definition}
Affine quasigroups are also called \emph{central quasigroups} or \emph{$T$-quasigroups} (\cite[Section 1.4.4.3]{Shch}). They are polynomially equivalent to a module over the ring of Laurent polynomials of two non-commuting variables.

If $(A,\ast)$ is an affine quasigroup, then the left division is defined by $a\backslash b=s^{-1}\big(b-ta-c_0\big)$ and the right division by $a/ b= t^{-1}\big(a-sb-c_0\big)$. 

\begin{remark}
A quasigroup $(Q,\cdot)$ is \emph{medial} (or \emph{entropic}) if it satisfies the medial law
\begin{equation}
(x\cdot y)\cdot (z\cdot t)=(x\cdot z)\cdot (y\cdot t).
\end{equation}
The famous Bruck-Murdoch-Toyoda Theorem \cite{Bruc,Mur,Toy} (see also \cite[Theorem 2.65]{Shch}) states that, up to isomorphism, medial quasigroups are precisely affine quasigroups with commuting automorphisms $t,s$.
\end{remark}

 In Section \ref{S:EH} we will find conditions on $t$, $s$ and $c_0$ under which an affine quasigroup satisfies a given BMq identity. We will see (check List \ref{list}), that $t$ and $s$ always commute (with the assumption that there are no zero divisors); thus we can think of $(A,\ast)$ as a module over $\mathbb{Z}[t,s]$. 

\section{Extensions and (co)Homology}\label{S:EH}

As briefly discussed in the introduction, our paper follows Eilenberg's lead on the relation between extensions and (co)homology in the case of Bol-Moufang quasigroups.
Let us cite from \cite{Eil}: {\it The theory of group extensions...has been shown (Eilenberg and Mac Lane \cite{EM12} to be
closely related with the so called cohomology theory of abstract groups. \ldots The purpose of this note is to present a theory
of extensions for some general classes of algebras including associative and Lie algebras as special cases. The factor-set
relations obtained should be useful in designing cohomology theories for other kinds of algebras.}
Our paper concentrates on a potential theory of (co)homology for Bol-Moufang quasigroups. We start by analyzing extensions of Bol-Moufang quasigroups
by affine quasigroups of the same type. From this we gain insight on $\partial_3$ and $\partial_2$ of a related chain complex and so we can define
$H_2$ and $H_1$ of Bol-Moufang quasigroups (concrete examples of calculations are given in Section \ref{S:examples}).

\subsection{Extensions for quasigroups with one relation}\label{S:EQR}
Quasigroups of Bol-Moufang type have one relation in addition to the quasigroup relations, and for quasigroups with one relation we can study extensions of such quasigroups by affine quasigroups with the same relation. The method is general and follows the philosophy of \cite{Eil}. That is, for a quasigroup $(X,*_X)$ and an affine quasigroup $(A,*_A)$, we construct an extension which is a quasigroup on $A\times X$ with an operation of the same BMq type as before. This will be explained in detail in an example of the right Bol quasigroup ($RBQ^L$). However all other cases are very similar so we leave it as an exercise to the reader (we summarize the results in List~\ref{list}).

As previously noted, we study extensions of Bol-Moufang quasigroups not to classify them, but to have a first glimpse at their homology via their second and third boundary operations.
\subsubsection{Extension of a quasigroup by an affine quasigroup}
As a warm up
we show that an extension of a quasigroup by an affine quasigroup is a quasigroup.

As we noted before if $a*b=ta+sb+c_o$ then
$$a/b= t^{-1}(a-sb-c_o)$$ and 
$$a\backslash b =s^{-1}(-ta+b-c_0).$$
Then we consider an extension of a quasigroup $X$ by an affine quasigroup $A$ given by 
$$(a,x)*(b,y)= (a*b+ \phi(x,y), xy).$$ 
Here $\phi$ is an arbitrary function, $\phi: X\times X \to A$, called 2-cochain; see Subsection~\ref{Ex:RBQ}, where we analyze conditions on $\phi$ imposed by a Bol-Moufang equation.
\begin{proposition}
An extension of a quasigroup by an affine quasigroup is a quasigroup.
\end{proposition}
\begin{proof}
We find that:
$$(a,x)/(b,y)= (a/b -t^{-1}\phi(x/y,y),x/y)$$
and analogously
$$(a,x)\backslash (b,y)=(a\backslash b - s^{-1}\phi(x,x\backslash y),x\backslash y).$$
We will check, as a typical case, that 
$((a,x)/(b,y))(b,y)= (a,x):$
$$((a,x)/(b,y))(b,y)=(a/b-t^{-1}\phi(x/y,y),x/y)*(b,y)= $$
$$((a/b)*b - t(t^{-1}\phi(x/y,y))+ \phi(x/y,y)),(x/y)y)=(a,x)$$
as needed. Other equalities for $\backslash $  and $/ $  are checked in a similar way.
\end{proof}

We used in our calculation the following simple lemma:
\begin{lemma}\label{lem:32}
We have the following identities in the affine quasigroup $(A,\ast)$ over an abelian group $(A,+)$:
\begin{enumerate}
\item[(1)] $(a+b)*c= ta+tb+sc+c_0= ta+sc+c_0 + tb= a*c+tb= b*c+ ta$.
\item[(2)] $ a*(b+c)= ta+sb +sc+ c_0= a*b+sc= a*c+sb$.
\item[(3)] $ (a+b)*(c+d)= ta+tb +sc+ sd+ c_0= a*c+tb+sd=a*c+b*d-c_0$.
\end{enumerate}
\end{lemma}

\subsubsection{Example: Calculating the cochain condition for right Bol quasigroups}\label{Ex:RBQ} 
The variety of right Bol quasigroups ($RBQ^L$) consists of quasigroups which satisfy the additional identity: \begin{equation}\label{eq:rbq}
x((yz)y) = ((xy)z)y.    
\end{equation}
(This is identity E25.) Let $X\in RBQ^L$. We analyze in this subsection extensions of $X$ by an affine quasigroup $A$ also satisfying $RBQ^L$ condition. We look for the $RBQ^L$ structure on $A\times X$ such that
$$(a,x)*(b,y)= (a*b +\phi(x,y);xy) $$
where $\phi: X\times X \to A$ is called a 2-cochain of the extension. We will find now the condition for which $(A\times X,*)$ belongs to $RBQ^L$. Thus we need: 
$$(a,x)*(((b,y) *(c,z))*(b,y)) = (((a,x)*(b,y))* (c,z))* (b,y) .$$
The calculation is as follows:
$$L=(a,x)*(((b,y)*(c,z))*(b,y)) = (a,x)*((b*c+\phi(y,z), y*z)*(b,y))=$$
$$(a,x)((b*c+\phi(y,z))*b+ \phi(yz,y), (yz)y)= (a,x)*((b*c)*b+t\phi(y,z)+ \phi(yz,y), (yz)y)=$$
$$(a*\big((b*c)*b+t\phi(y,z)+ \phi(yz,y)\big)+ \phi(x,(yz)y), x((yz)y)=$$
$$(a*((b*c)*b)+st\phi(y,z)+ s\phi(yz,y))+ \phi(x,(yz)y), x((yz)y)).$$
Similarly we get:
$$R=(((a,x)*(b,y))* (c,z))* (b,y)=((a*b+\phi(x,y),xy)*\big(c,z)\big)*(b,y)=$$
$$((a*b+\phi(x,y))*c+ \phi(xy,z)),(xy)z))*(b,y)=((a*b)*c+t\phi(x,y)+\phi(xy,z))),(xy)z)*(b,y)=$$
$$(((a*b)*c+t\phi(x,y)+\phi(xy,z))*b+\phi((xy)z,y), ((xy)z)y))=$$
$$(((a*b)*c)*b+ t^2\phi(x,y)+t\phi(xy,z))+\phi((xy)z,y),((xy)z)y))$$
 Thus $L=R$ if 
 $$st\phi(y,z)+ s\phi(yz,y)+ \phi(x,(yz)y)= t^2\phi(x,y) + t\phi(xy,z) +\phi((xy)z,y).$$
 We extend $\phi\colon X\times X\to A$ linearly to $\phi\colon k(X\times X)\to A$, where $k$ is a commutative ring with invertible elements $t,s$ (we can take as $k$ a ring of polynomials of 2 variables $\Z[t,s])$. Hence, 
 $$\phi\big(st(y,z)+s(yz,y)+(x,(yz)y)\big)=\phi\big(t^2(x,y) + t(xy,z)+((xy)z,y)\big).$$
 Following examples of groups (or semigroups) and quandles (or shelves)
 we consider the part of a chain complex $C_n=kX^n$ with
 $$\partial_3(x,y,z)=st(y,z)+s(yz,y)+(x,(yz)y)- t^2(x,y) - t(xy,z)-((xy)z,y),$$
 here $\partial_3(x,y,z)\colon kX^3\to kX^2$ (see Subsections \ref{S:homologyboundary} and \ref{S:boundarymaps}).

We also check for which $t$, $s$ and $c_0$ the affine quasigroup satisfies the given Bol-Moufang condition. For example for $RBQ^L$ we need to have $a*((b*c)*b)=((a*b)*c)*b$.\footnote{The straightforward calculation is as follows:
$$a*((b*c)*b)=a*((tb+sc+c_0)*b)=a*(t^2b+tsc+tc_0++sb+c_0)=$$
$$ta+st^2b+stsc+stc_0+s^2b+sc_0+c_0=ta+(st^2+s^2)b+stsc+(1+s+st)c_0.$$
$$((a*b)*c)*b= ((ta+sb+c_0)*c)*b=(t^2a+tsb+tc_0+ sc+c_0)*b=$$
$$t^3a+t^2sb+t^2c_0+tsc+tc_0+sb+c_0=t^3a+(t^2s+s)b+tsc+(1+t+t^2)c_0.
$$
$$a*((b*c)*b)-((a*b)*c)*b= (t-t^3)a+(st^2+s^2-t^2s-s)b+(sts-ts)c+(s+st-t-t^2)c_0;$$
Compare example (11) of List \ref{list}.
This is done, for every Bol-Moufang quasigroup case in List \ref{list} and for these we use expression $\hat H(T)$ of Definition \ref{def:polynomials} (H). }

\subsubsection{Equivalent extensions and $\partial_2$} 

We consider in this subsection the conditions under  which two extensions by an affine quasigroups are equivalent. This will allow us to recover $\partial_2(x,y): kX^2\to kX$.
\begin{definition} 
We say that two extensions of $X$ by $A$, where $(A\times X,*_1)$ is given by a $2$-cochain $\phi_1$ and $(A\times X,*_2)$ is given by a $2$-cochain  $\phi_2$ are equivalent if there is a $1$-cochain $\alpha: X \to A$ such that the map 
$\hat\alpha: (A\times X,*_1)\to (A\times X,*_2)$ given by 
$$\hat\alpha(a,x)= (a+\alpha(x),x)$$
is a quasigroup homomorphism. That is:
$$\hat\alpha((a_1,x_1)*_1(a_2,x_2))=\hat\alpha(a_1,x_1)*_2\hat\alpha(a_2,x_2)$$
\end{definition}
We will show that the condition for a homomorphism (in fact an isomorphism) $\hat\alpha$ leads to the second boundary map $\partial_2$ given by:
$$\partial_2(x,y)= tx + sy - xy.$$
See Section~\ref{partialsquared} for a proof that $\partial_2 \partial_3 = 0$. 

Recall that in $A$ we have $a*b= ta+sb+c_0.$

\begin{proposition}\label{Second}
Two equivalent extensions $(A\times X,*_1)$ and  $(A\times X,*_2)$ given by $2$-cochains $\phi_1$ and $\phi_2$ and related using isomorphism $\hat\alpha$ lead to the equation:
$$\phi_1(x_1,x_2)- \phi_2(x_1,x_2)=\alpha(tx_1+sx_2-x_1x_2).$$

\end{proposition}
\begin{proof}
We compute
$$\hat\alpha((a_1,x_1)*_1(a_2,x_2))= \hat\alpha(a_1*a_2+\phi_1(x_1,x_2),x_1x_2)=$$
$$(a_1*a_2+ \alpha(x_1x_2)+ \phi_1(x_1,x_2),x_1x_2)$$

$$\hat\alpha(a_1,x_1)*_2\hat\alpha(a_2,x_2)=(a_1+\alpha(x_1),x_1)*_2 (a_2+\alpha(x_2),x_2)= $$
$$((a_1+\alpha(x_1))*_2(a_2+\alpha(x_2))+ \phi_2(x_1,x_2),x_1x_2)\stackrel{Lemma \ref{lem:32}}{=}$$
$$(a_1*a_2+ t\alpha(x_1)+s\alpha(x_2)+ \phi_2(x_1,x_2),x_1x_2).$$
For these two expressions to be equal we need:
$$\phi_1(x_1,x_2)-\phi_2(x_1,x_2)= t\alpha(x_1) +s\alpha(x_2)- \alpha(x_1x_2).$$
Equivalently we can write it as:
$$\phi_1(x_1,x_2)-\phi_2(x_1,x_2)= \alpha(tx_1+sx_2- x_1x_2)$$ 
and from this we deduce that
$\partial_2(x_1,x_2)= tx_1+sx_2-x_1x_2$.

\end{proof}

\subsection{Substitutions, Rooted Binary Trees} \label{sec:substitutions}
As a preparation for extension properties and related homology we introduce some invariants of rooted binary trees.

\begin{definition}\label{def:polynomials}
Consider a rooted binary tree with $n$ leaves enumerated from left to right, $1,2,...,n$ (compare Example \ref{E:tree}). We associate to such a tree with ordered leaves three functions, $h(T),H(T)$ and $Q(T)$, as follows.
\begin{enumerate}
\item[(h)] $h(T)=\sum_{v\in V(T)-L(T)} h_v(t,s)$,  where $V(T)$ is the set of vertices of $T$ and $L(T)$ is the set of leaves, and $h_v(t,s)$ is the associated word\footnote{
More formally, we define $h_v(t,s)$ inductively by $h_{root}(t,s)=1$ and then recursively, if $v_L$ is the left child of $v$ and $v_R$ the right child of $v$ then we define $h_{v_L}(t,s)= h_v(t,s) t$ and $h_{v_R}(t,s)=sh_v(t,s)$.} in letters $t$ and $s$ obtained by following the unique path from the root to $v$ and creating the word by putting $t$ if the path turns left and $s$ if it turns right. If $T$ is one vertex tree then we put $h(T)=0.$
\item[(H)] $H(T)= \sum_{v_i\in L(T)}h_{v_i}(t,s)a_i$ while $n$ elements $a_1,a_2,...a_n$ are chosen from an affine quasigroup $A$ ($n$ is the number of leaves of $T$)  and 
label leaves of $T$ by these letters from left to the right. We let letters $x_i$ and $a_i$ label the same leaf and there is a bijection between $x_i$ and $a_i$ sequences.
We define $H(T)= \sum_{v_i\in L(T)}h_{v_i}(t,s)a_i$. 
If $T$ is one vertex tree then we put $H(T)=a_1.$

We combine $H(T)$ and $h(T)$  as follows
$$\hat H(T)= H(T)+h(T)c_0 \mbox{ for a fixed $c_0$}.$$ 

This expression is used to determine for which $t$, $s$, and $c_0$ an affine quasigroup is of given Bol-Moufang type; see Lemma \ref{lem:H}.

\item[(Q)] We first associate recursively to each 
vertex of $T$ a weight $\bar g(v)$ being a word in alphabet 
$x_1,x_2,...,x_n$ that is:
\[
\bar g(v)=
\left\{
\begin{array}{ll}
x_i & \mbox{when $v\in L(T)$ is decorated by $x_i$},\\
\bar g(v_L)\bar g(v_R) & \mbox{when $v\in V(T)-L(T)$ and $v_L$, $v_R$ are children of $v$ in $T$}.
\end{array}
\right.
\]
Having $\bar g(v)$ defined\footnote{We associated to 
every vertex $v$ of a tree $T$ a weight $\bar g(v)$, we described it recursively but here we stress that $T$ was created from some word in $X$ so now every vertex has some tree under $v$, say $T_v$, and $\bar g(v)$ is the word of $T_v$.} we define $g(v)$ for $v\in V(T)-L(T)$ as a pair 
$(\bar g(v_L),\bar g(v_R))$.
 We define now $Q(T)$ as the sum over 
all vertices except leaves of $h_v(t,s)g(v)$ that is 
$$Q(T)=\sum_{v\in V(T)-L(T)} h_v(t,s)g(v).$$
Notice that $Q(T)= h(T)$ if we put $g(v)=1$.
\end{enumerate}
\end{definition}
Note that in Section~\ref{S:EQR}, we have defined $\partial_3(x,y,z)$ for right Bol quasigroups (E25) as $Q(E2) - Q(E5)$. In general, we take $\partial_3(x,y,z) = Q(Vi) - Q(Vj)$ for quasigroups satisfying identity $Vij$. 

In the case of Bol-Moufang quasigroups we will have binary rooted trees of 4 leaves and one variable associated to leaves in $H(T)$ and $Q(T)$ will be repeated.

\begin{example}\label{E:tree}
 For the rooted tree $T$ presented in Figure \ref{rootedtree1} (corresponding to the word $((x_1(x_2x_3))(x_4x_5))(x_6x_7)$), one obtains the following polynomials:
\begin{align*}
h(T) =& 1+t+s+t^2+ts+t^2s,\\
H(T)=&t^3a_1+ t^2sta_2+ t^2s^2a_3+ tsta_4+ ts^2a_5+sta_6+s^2a_7,\\ 
Q(T)=&t^2s(x_2,x_3)+ ts(x_4,x_5)+t^2(x_1,x_2x_3)+  s(x_6,x_7)+\\
 &t(x_1(x_2x_3),x_4x_5) + ((x_1(x_2x_3))(x_4x_5),x_6x_7).
\end{align*}
\setlength{\unitlength}{1mm}
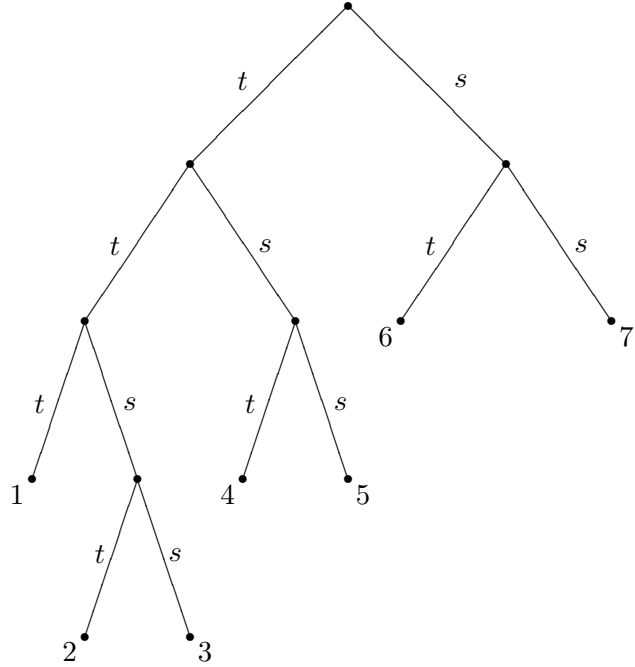
\begin{figure}[hbt]
    \centering
 \begin{picture}(90,90)(0,0)
\put(7,23){\makebox(0,0){$1$}} \put(9,25){\circle*{1}}
\put(14,2){\makebox(0,0){$2$}} \put(16,4){\circle*{1}}
\put(32,2){\makebox(0,0){$3$}} \put(30,4){\circle*{1}}
\put(16,46){\circle*{1}}
\put(23,25){\circle*{1}}
\put(30,67){\circle*{1}}
\put(51,88){\circle*{1}}
\put(9,25){\line(1,3){7}} \put(10,35){\makebox(0,0){$t$}}
\put(16,46){\line(2,3){14}} \put(20,56){\makebox(0,0){$t$}}
\put(16,46){\line(1,-3){7}} \put(22,35){\makebox(0,0){$s$}}
\put(23,25){\line(-1,-3){7}} \put(18,15){\makebox(0,0){$t$}}
\put(23,25){\line(1,-3){7}} \put(28,15){\makebox(0,0){$s$}}
\put(30,67){\line(1,1){21}} \put(37,78){\makebox(0,0){$t$}}
\put(44,46){\circle*{1}}
\put(30,67){\line(2,-3){14}} \put(40,56){\makebox(0,0){$s$}}
\put(35,23){\makebox(0,0){$4$}} \put(37,25){\circle*{1}}
\put(53,23){\makebox(0,0){$5$}} \put(51,25){\circle*{1}}
\put(37,25){\line(1,3){7}} \put(38,35){\makebox(0,0){$t$}}
\put(51,25){\line(-1,3){7}} \put(50,35){\makebox(0,0){$s$}}
\put(51,88){\line(1,-1){21}} \put(66,78){\makebox(0,0){$s$}}
\put(72,67){\circle*{1}}
\put(56,44){\makebox(0,0){$6$}} \put(58,46){\circle*{1}}
\put(88,44){\makebox(0,0){$7$}} \put(86,46){\circle*{1}}
\put(58,46){\line(2,3){14}} \put(62,56){\makebox(0,0){$t$}}
\put(86,46){\line(-2,3){14}} \put(82,56){\makebox(0,0){$s$}}
\end{picture}   
\caption{A rooted tree $T$.}
\label{rootedtree1}
\end{figure}
\end{example}

\begin{example}\label{E:fig1_5}

\noindent For the rooted trees of Figures \ref{F:bracket_1} through \ref{F:bracket_5}, we have the following polynomials: 

\begin{center}
{\small 
\begin{tabular}{c|c|c}
    $h$ & $H$ & $Q$ \\
    \hline
    $1 + s + s^2$ & $t a_1 + st a_2 + s^2t a_3 + s^3 a_4$ & $s^2(x_3, x_4) + s(x_2, x_3x_4) + (x_1,x_2(x_3x_4))$ \\
    $1 + s + st$ & $t a_1 + st^2 a_2 + sts a_3 + s^2 a_4$ & $st(x_2, x_3) + s(x_2x_3, x_4) + (x_1, (x_2x_3)x_4)$ \\
    $1 + t + s$ & $t^2 a_1 + ts a_2 + st a_3 + s^2 a_4$ & $t(x_1,x_2) + s(x_3,x_4) + (x_1x_2, x_3x_4)$ \\
    $1 + t + ts$ & $t^2 a_1 + tst a_2 + ts^2 a_3 + s a_4$ & $ts(x_2, x_3) + t(x_1, x_2x_3) + (x_1(x_2x_3), x_4)$ \\
    $1 + t + t^2$ & $t^3 a_1 + t^2s a_2 + ts a_3 + s a_4$ & $t^2(x_1, x_2) + t(x_1x_2, x_3) + ((x_1x_2)x_3, x_4)$
\end{tabular}}
\end{center}
\end{example}
As mentioned at the end of part (H) of Definition \ref{def:polynomials}, we can use expression $\hat H$ to analyze affine quasigroups with specific relations:
\begin{lemma}\label{lem:H}
Let $w=a_1*a_2*...*a_n$, with chosen bracketing given by a rooted binary tree $T_w$, be a word in the affine quasigroup $(A,*)$ over the abelian group
$(A,+)$. If we write $w$ as a linear combination of $a_1,...,a_n,c_0$ and use the definition of $*$ in the affine case, then
$$w= H(T_w)+h(T_w)c_0 \mbox{ or shortly } w=\hat H(T_w).$$
We use this lemma to find conditions for $t$, $s$ and $c_0$ for which the affine quasigroup $(A,*)$ satisfies the equation $w_1=w_2$ for chosen $w_1$ and $w_2$. In List \ref{list} we use this lemma for
Bol-Moufang affine quasigroup equations defined in \cite{PhVo1}.
\end{lemma}
\begin{proof}
We proceed by induction on $n$. The initial condition, for $n=1$, holds immediately as in the case $T$ has one vertex, labelled $a_1$, so $w=a_1$ and $\hat H(T)=a_1$ by definition.
For an inductive step, we assume that Lemma \ref{lem:H} holds for word shorter than $n$ ($n>1$), and we consider $w=w_L*w_R$ where $w$ is of length $n$ and $w_L$ and $w_R$ are shorter. 
Thus $w=tw_L+sw_R +c_0$ and by the inductive assumption $w_L=\hat H(T_{w_L})$ and $w_R= \hat H(T_{w_R})$. Therefore
\begin{align*}
    w & = t\hat H(T_{w_L})+s\hat H(T_{w_R})+ c_0 \\
    & = tH(T_{w_L})+sH(T_{w_R}) + th(T_{w_L})c_0+sh(T_{w_R})c_0 +c_0 \\
    & = H(T_w) +h(T_v)c_0= \hat H(T_w) \mbox{ as needed.}
\end{align*}
\end{proof}

\begin{List}\label{list}
    The following list gives $H$ and $h$ for each variety of quasigroups of Bol-Moufang type. If $Vij$ is the additional identity defining the variety, we calculate $H(Vij)$ as $H(Vi)-H(Vj)$, where $Vi$ and $Vj$ are the rooted binary trees corresponding to the words $Vi$ and $Vj$, respectively. In a similar way, we obtain $h(Vij)$. It allows us to know the structure (values of $t$, $s$ and $c_0$) for affine Bol-Moufang quasigroups in each of 26 cases described in \cite{PhVo1}. In this list we use $h$ and $H$ to find the values of $t$ and $s$ for which our definition of $a * b$ on $A \times X$ gives an affine quasigroup satisfying $Vij$. Because we work with affine quasigroups, $t$ and $s$ are invertible, however initially we do not assume that $t$ and $s$ commute, but it follows in each case with the assumption that the ring we work in has no zero divisors. Thus $H$ and $h$ are  computed without assuming commutativity which follows in every case from the fact that each coefficient of $H$ should be equal to zero.
    Furthermore notice that for a right loop we have $t=1$ and for a left loop $s=1$.
    Moreover, we check that the conditions on $t$ and $s$ do not depend on the identity used to define the corresponding variety of quasigroups.
     After \cite{PhVo1} we have a concept of duality to BMq with concrete identity. We can now say that given $Vij$ Bol-Moufang quasigroup is dual to $V'j'i'$ (see Subsection \ref{Ss:dual}) if they are defined using some (so all) dual identities. In this language we can also say when BMq is selfdual. Note that we can prove the commutativity of $t,s$ without assuming no zero divisors in all but 4 cases: $A14,$ its dual $ F25$, $B23$, and $C15$. 

\bigskip

An example entry looks like the following:
     \begin{enumerate}
        \item[(1)] Name of the Quasigroup Variety
        \begin{enumerate}
            \item \textbf{Defining Relation $Vij$}
            \begin{enumerate}
                \item $H(Vij)$
                \item $h(Vij)$
                \item Commutativity of $t$ and $s$. (Note that $t$ and $s$ commute in every case under the assumption that there are no zero divisors, but this assumption is only \textit{required} in a few cases: A14, B23, F25, and C15.
                \item Solutions for $s,t$.
                \item Restrictions on $c_0$
            \end{enumerate}
        \end{enumerate}
    \end{enumerate}
\end{List}

\bigskip

\bigskip

    \begin{enumerate}

        \item[(1)] Groups (so also loops): $t=s=1$ in these cases and $c_0$ is arbitrary.
        \item[(2)] $RG1^L$
        \begin{enumerate}
            \item \textbf{$A25$} (dual to $F14$)
            \begin{enumerate}
                \item $H(A25) = (t(1-t^2-ts)+st^2)a+(s-1)tsb+s(s-1)c$
                \item  $h(A25)=(s-t)(1+t)$
                \item Commuting $t$ and $s$ without additional assumptions since $s=1$.
                \item Solutions: $s=1$, $t^2 = 1$.
                Assuming no zero divisors, $(t,s) = (\pm 1, 1)$.
                \item No restrictions on $c_0$
            \end{enumerate}
           
            \item \textbf{$D25$} (dual to $D14$)
            \begin{enumerate}
                \item  $H(D25) = (t+s^2-t^3-s)a + (st^2 -t^2s)b +(s-1)tsc$
                \item  $h(D25) = s+st-t-t^2=(s-t)(1+t)$.
                \item Commuting $t$ and $s$ without additional assumptions.
                \item Solutions are the same as $A25$.
                \item No restrictions on $c_0$.
            \end{enumerate}
        \end{enumerate}

        \item[(3)] $LG1^R$
        \begin{enumerate}
            \item \textbf{$F14$} (dual to $A25$)
            \begin{enumerate}
                \item $H(F14) = t(1-t)a+(1-t)stb+(s^2t+s^3-ts^2-s)c$
                \item  $h(F14) = s+s^2 -t -ts = (s - t)(1 + s)$
                \item Commuting $t$ and $s$ without additional assumptions. By condition $(a)$, $t^2=1$. If $t=1,$ $(c)$ yields $s^2=1$ and $(b)$ is satisfied. If $t=-1$, then $(c)$ is $-s^2 + s^3 + s^2 -s$, so the same solution set.
                \item Solutions: $t=1$, $s^2=1$. Assuming no zero divisors, $(t,s) = (1, \pm 1)$.
                \item No restrictions on $c_0$
            \end{enumerate}
           
            \item \textbf{$D14$} (dual to $D25$)
            \begin{enumerate}
                \item  $H(D14) = (t+s^3-t^2-s)a +(1-t)stb +(s^2t-ts^2)c$
                \item   $h(D14) = s+s^2 -t -ts = (s - t)(1 + s)$
                \item Commuting $t$ and $s$ without additional assumptions since $t=1$.
                \item Solution: $t=1$, $s^2 = 1$. Assuming no zero divisors, $(t,s) = ( 1, \pm 1)$ 
                \item No restrictions on $c_0$
            \end{enumerate}
        \end{enumerate}

        \item[(4)] $RG2^L$
        \begin{enumerate}
            \item \textbf{$A23$} (dual to $F34$)
            \begin{enumerate}
                \item $H(A23)= (t+st^2-t^2-ts)a +st(s-1)b $
                \item  $h(A23) = (s-1)t$
                \item Commuting $t$ and $s$ without additional assumptions since $s=1$.
                \item Solution: $s=1$ and $t$ is arbitrary.
                \item No restrictions on $c_0$
            \end{enumerate}
        \end{enumerate}

 \item[(5)] $LG2^R$
        \begin{enumerate}
            \item \textbf{$F34$} (dual to $A23$)
            \begin{enumerate}
                \item  $H(F34) = ts(1-t)b +(st+s^2-ts^2-s)c $
                \item  $h(F34) = (1-t)s$
                \item Commuting $t$ and $s$ without additional assumptions since $t=1$.
                \item Solutions: $t=1$ and $s$ is arbitrary.               \item No restrictions on $c_0$.
            \end{enumerate}
        \end{enumerate}

 \item[(6)] $RG3^L$
        \begin{enumerate}
            \item \textbf{$B25$} (dual to $E14$)
            \begin{enumerate}
                \item  $H(B25) =(t+sts-t^3-ts)a+(st^2-t^2s)b+s(s-1)c $
                \item  $h(B25) = s+st-t-t^2 = (s-t)(1+t)$
                \item Commuting $t$ and $s$ without additional assumptions since $s=1$.
                \item Solution: $t^2=1$, $s=1$. Assuming no zero divisors, $(t,s) = (\pm 1, 1)$.
                \item No restrictions on $c_0$.
            \end{enumerate}
        \end{enumerate}

 \item[(7)] $LG3^R$
        \begin{enumerate}
            \item $E14$ (dual to $B25$)
            \begin{enumerate}
                \item   $H(E14) = t(1-t)a+(st+s^3-tst-s)b+(s^2t-ts^2)c$
                \item  $h(E14) = s+s^2-t-st= (s-t)(1+s)$ 
                \item Commuting $t$ and $s$ without additional assumptions since $t=1$.
                \item Solution: $t=1$, $s^2 = 1$. Assuming no zero divisors, $(t,s) = ( 1, \pm 1)$
                \item No restrictions on $c_0$.
                \item Note:  if $(G,\cdot)$ is a group then $(G,/ )$ where $a/b=ab^{-1}$ is an $LG3^R$ quasigroup.
            \end{enumerate}
        \end{enumerate}

\item[(8)] $EQ^2$
        \begin{enumerate}
            \item \textbf{$B23$} (dual to $E34$)
            \begin{enumerate}
                \item $H(B23) = (t+sts-t^2-st)a+(st^2-ts)b$
                \item $h(B23) = (s-1)t$
                \item Commuting $t$ and $s$ only by assuming no zero divisors. To see this, note that the coefficient of $b$ should be equal to $0$---we call this ``condition $(b)$" for short---so  $st^2 =ts$. Now, we rewrite condition $(a)$ as $0=t+s^2t^2-t^2-st = (1-s)t+(s^2-1)t^2= (1-s)(1-(1+s)t)t.$ Here, we assume no zero divisors and conclude either $(1-s)=0$ or $(1-(1+s)t)=0$. If $1-s=0$, then $t=s=1$. If $(1-(1+s)t)=0$, then we see that 
                \begin{align*}
                  \hspace{3 cm}  &  1 = (1+s)t \\
                    & \iff t = t^2+st^2  & \text{multiply by $t$}\\
                    & \iff t = t^2 + ts & \text{$st^2=ts$ by condition $(b)$} \\
                    & \iff t+s=1 \\
                    & \iff s=(1-t). 
                \end{align*}
                Thus, $s$ and $t$ commute. In condition $(b)$ this gives $st^2 = st$, implying $t = 1$. However, $s = (1 - t) = 0$, contradicting the assumption that $s$ is invertible. 
                \item Solution: since $s = 1 - t$ does not lead to a valid solution, the only possibility is $t=s=1$.
                \item No restrictions on $c_0$.
            \end{enumerate}
                
        \item $D15$ (self-dual)
            \begin{enumerate}
                \item  $H(D15) = (t+s^3-t^3-s)a + (st-t^2s)b +(s^2t-ts)c $
                \item $h(D15)=(s-t)(1-s-t)$ 
                \item Commuting $t$ and $s$ without additional assumptions. To see this, note that by substituting condition $(b)$ into $(c)$, $st^2s=ts$ so $st=1$. By substituting $(c)$ into $(b)$, we get $ts=1$.
                \item Solution: $t=s=1$. 
                \item No restrictions on $c_0$.
            \end{enumerate}
            
\item $E34$ (dual to $B23$)
            \begin{enumerate}
                \item $H(E34)+(ts+s^2-tst-s)b + (st - ts^2)c$
                \item  $h(E34) = (1-t)s$
                \item Commuting $t$ and $s$ without additional assumptions.  
                \item Solutions: $t=s=1$.
                \item No restrictions on $c_0$.
            \end{enumerate}
        \end{enumerate}

 \item[(9)] $MQ^2$
        \begin{enumerate}
            \item $B15$ (dual to $E15$)
            \begin{enumerate}
                \item  $H(B15) = (t+s^2t-t^3-ts)a+(st-t^2s)b +s(s^2-1)c$
                \item  $h(B15) = s+s^2 -t-t^2 = (s-t)(1+s+t)$ 
                \item Commuting $t$ and $s$ without additional assumptions. To see this, note that by condition $(c)$, $s^2 = 1$. Substituting this expression into $(a)$, we get 
                \begin{align*}
                    0 & = t+(s^2)t-t^3-ts \\
                    & = t+t-t^3-ts \\
                    & =t(2-t^2-s) \\
                \end{align*}
                Therefore, $s = 2-t^2$  so $t$ and $s$ commute. 
                \item Solution: $t=s=1$
                \item No restrictions on $c_0$.
            \end{enumerate}

             \item $D23$ (dual to $D34$)
            \begin{enumerate}
                \item  $H(D23) = (t-t^2)a + (st^2 -ts)b +st(s-1)c $
                \item   $h(D23) = s+st -t -s = (s-1)t$
                \item Commuting $t$ and $s$ without additional assumptions since $s=t=1$.
                \item Solution: $t=s=1$.
                \item No restrictions on $c_0$.
            \end{enumerate}

             \item $E15$ (dual to $B15$)
            \begin{enumerate}
                \item  $H(E15) = t(1-t^2)a + (st+s^3 -t^2s-s)b +(s^2t-ts)c$
                \item   $h(E15) = s+s^2 -t -t^2 $
                \item Commuting $t$ and $s$ without additional assumptions since $t=s=1$.
                \item Solution: $t=s=1$
                \item No restrictions on $c_0$.
            \end{enumerate}
        \end{enumerate}

 \item[(10)] $LBQ^R$
        \begin{enumerate}
            \item $B14$ (dual to $E25$)
            \begin{enumerate}
                \item  $H(B14) = (t+s^2t-t^2-ts^2)a+(1-t)stb+s(s^2-1)c$
                \item  $h(B14) = s+s^2-t-ts$ 
                \item Commuting $t$ and $s$ without additional assumptions since $t=1$.
                \item Solution: $t=1$, $s^2 = 1$. Assuming no zero divisors, $(t,s) = ( 1, \pm 1)$.
                \item No restrictions on $c_0$.
            \end{enumerate}
        \end{enumerate}

 \item[(11)] $RBQ^L$
        \begin{enumerate}
            \item $E25$ (dual to $B14$)
            \begin{enumerate}
                \item $H(E25)= t(1-t^2)a+(st^2+s^2-t^2s-s)b+(s-1)tsc $
                \item $h(E25) = s+st -t -t^2=(s-t)(1+t)$
                \item Commuting $t$ and $s$ without additional assumptions since $s=1$.
                \item Solutions: $t^2 = 1$, $s=1$. Assuming no zero divisors, $(t,s) = ( \pm 1, 1)$
                \item No restrictions on $c_0$.
            \end{enumerate}
        \end{enumerate}

 \item[(12)] $CQ^0$
        \begin{enumerate}
            \item $C15$ (self-dual)
            \begin{enumerate}
                \item $H(C15) = t(1-t^2)a+(st^2 + s^2 - t^2 s - s)b + (s-1)ts c$
                \item $h(E25) = s+st-t-t^2 = (s-t)(1+t)$ 
                \item Commuting $t$ and $s$ assuming no zero divisors. 
                \item Solutions: $t=s=1$ and $t=s=-1$. Notice that there is a solution with neither $t$ nor $s$ equal to 1, which allows us to easily conclude that $C15$ is not a right/left loop.
                \item No restrictions on $c_0$.
            \end{enumerate}
        \end{enumerate}

\item[(13)] $LC1^2$
        \begin{enumerate}
            \item $A34$ (dual to $F23$)
            \begin{enumerate}
                \item     $H(A34) = ts(1-t)a +(st-ts^2)b+s(s-1)c$
                \item $h(A34) = s(1-t)$
                \item Commuting $t$ and $s$ without additional assumptions since $t=s=1$.
                \item Solution: $t=s=1$.
                \item No restrictions on $c_0$.
            \end{enumerate}
        \end{enumerate}

\item[(14)] $LC2^0$
        \begin{enumerate}
            \item $A14$ (dual to $F25$)
            \begin{enumerate}
                \item  $H(A14)= (t+st-t^2-tst)a+ (s^2t-ts^2)b+ (s^3-s)c$
                \item $h(A14)=(1+s+s^2) - (1+t+ts)=(s-t)(s+1)$
                \item Commuting $t$ and $s$ assuming no zero divisors. By condition $(c)$, $s^2 = 1$, thus condition $(b)$ is satisfied automatically. We rewrite condition $(a)$ as $(1-t)(1+s)t = 0$ and by assuming no zero divisors, conclude that $t = 1$ or $s=-1$. 
                \item Solutions: $t=s=1$ is a solution as well as $s=-1$ with $t$ arbitrary. 
                \item No restrictions on $c_0$.
            \end{enumerate}
        \end{enumerate}

        \item[(15)] $LC3^L$
        \begin{enumerate}
            \item $A15$ (dual to $F15$)
            \begin{enumerate}
                \item $H(A15) = t(1+s-t^2-ts)a+(s^2t-ts)b+s(s^2-1)c$
                \item $h(A15) = s+s^2-t-t^2=(s-t)(1+s+t) $
                \item Commuting $t$ and $s$ without additional assumptions.
                \item Solutions: $s=1$ and $t^2+t = 2$. Assuming no zero divisors, $(t,s) = (-2,1)$
                \item No restrictions on $c_0$.
            \end{enumerate}
        \end{enumerate}

        \item[(16)] $LC4^R$
        \begin{enumerate}
            \item $C14$ (dual to $C25$)
            \begin{enumerate}
                \item $H(C14) = t(1-t)a+(st+s^2t-tst-ts^2)b+s(s^2-1)c$
                \item $h(C14) = s+s^2-t-ts = (s-t)(1+s)0$
                \item Commuting $t$ and $s$ without additional assumptions.
                \item Solutions: $t=1$, $s^2 = 1$. Assuming no zero divisors, $(t,s) = ( 1, \pm 1)$.
                \item No restrictions on $c_0$.
            \end{enumerate}
        \end{enumerate}

        \item[(17)] $RC1^2$
        \begin{enumerate}
            \item $F23$ (dual to $A34$)
            \begin{enumerate}
                \item $H(F23) = t(1-t)a+(st^2-ts)b+st(s-1)c$
                \item $h(F23) = (s-1)t$
                \item Commuting $t$ and $s$ without additional assumptions.
                \item Solution: $t=s=1$ 
                \item No restrictions on $c_0$.
            \end{enumerate}
        \end{enumerate}

        \item[(18)] $RC2^0$
        \begin{enumerate}
            \item $F25$ (dual to $A14$)
            \begin{enumerate}
                \item $H(F25)=t(1-t^2)a + (st^2-t^2s)b+ (s-1)(t+1)sc$
                \item $h(F25)=(1+s+st)-(1+t+t^2)=(s-1)(1+t)$
                \item Commuting $t$ and $s$ assuming no zero divisors. The necessity of this assumption is clear from the duality of $F25$ to $A14$.
                \item Solutions: $t=s=1$ is one solution and the other is $t=-1$ with $s$ arbitrary.
                \item No restrictions on $c_0$.
            \end{enumerate}
        \end{enumerate}

        \item[(19)] $RC3^R$
        \begin{enumerate}
            \item $F15$ (dual to $A15$)
            \begin{enumerate}
                \item $H(F15) = t(1-t^2)a+(st-t^2s)b+(s^2t+s^3-ts-s)c$
                \item  $h(F15) = s+s^2-t-t^2=(s-t)(1+s+t)$
                \item Commuting $t$ and $s$ assuming no zero divisors. 
                \item Solutions: Assuming no zero divisors, $t=1$ and $s$ is either $1$ or $-2$. 
                \item No restrictions on $c_0$.
            \end{enumerate}
        \end{enumerate}

\item[(20)] $RC4^L$
        \begin{enumerate}
            \item $C25$ (dual to $C14$)
            \begin{enumerate}
                \item $H(C25) = t(1-t^2)a +(st^2+sts-t^2s-ts)b +s(s-1)c$
                \item  $h(C25) = s+st -t-t^2= (s-t)(1+t)$
                \item Commuting $t$ and $s$ without additional assumptions since $s=1$.
                \item Solutions: $t^2 =1$ and $s=1$. Assuming no zero divisors, $(t,s) = (\pm 1, 1)$.
                \item No restrictions on $c_0$.
            \end{enumerate}
        \end{enumerate}

\item[(21)] $LAQ^L$
        \begin{enumerate}
            \item $A13$ (dual to $F35$)
            \begin{enumerate}
                \item  $H(A13) = (t+st-t^2-ts)a+(s-1)stb+s^2(s-1)c$
                \item  $h(A13) = s^2-t$
                \item Here, $s=1$ so $t$ and $s$ commute without additional assumptions.
                \item Solution: $t=s=1$ 
                \item No restrictions on $c_0$.
            \end{enumerate}

             \item $A45$ (dual to $F12$)
            \begin{enumerate}
                \item  $H(A45) = t(st + t - t^2 -ts)a + ts(s - 1) b$
                \item  $h(A45) = t+ts-t-t^2$
                \item Here, $s=t=1$ so $t$ and $s$ commute without additional assumptions.
                \item Solutions: $(t,s) = ( 1, 1)$ 
                \item No restrictions on $c_0$.
            \end{enumerate}

             \item $C12$ (dual to $C45$)
            \begin{enumerate}
                \item  $H(C12) = s(t+st-t^2 -ts))b +s^2(s-1)c$
                \item     $h(C12) = s(s-t)$
                \item Commuting $t$ and $s$ without additional assumptions since $t=s=1$.
                \item Solutions: $(t,s) = ( 1, 1)$ 
                \item No restrictions on $c_0$.
            \end{enumerate}
        \end{enumerate}

        \item[(22)] $RAQ^R$
        \begin{enumerate}
            \item $C45$ (dual to $C12$)
            \begin{enumerate}
                \item $H(C45) = t^2(1-t)a+t(st+s^2-ts-s)b$
                \item $h(C45) = t(s-1)$
                \item Commuting $t$ and $s$ without additional assumptions since $t=s=1$.
                \item Solution: $t=s=1$ 
                \item No restrictions on $c_0$.
            \end{enumerate}

            \item $F12$ (dual to $A45$)
            \begin{enumerate}
                \item $H(F12) = (t-t^2)a +st(1-t)b + (s^2t+s^3-ts^2-s^2)c$
                \item  $h(F12) = s +s^2 -t -ts =(s-t)(1+s)$
                \item Commuting $t$ and $s$ without additional assumptions since $t=s=1$.
                \item Solution: $t=s=1$  
                \item No restrictions on $c_0$.
            \end{enumerate}

        \item $F35$ (dual to $A13$)
            \begin{enumerate}
                \item $H(F35) = t^2(1-t)a +(1-t)tsb + (st+s^2 -ts -s) c $
                \item  $h(F35) = s -t^2$
                \item Commuting $t$ and $s$ without additional assumptions since $t=s=1$.
                \item Solution: $t=s=1$  
                \item No restrictions on $c_0$.
            \end{enumerate}
        \end{enumerate}

\item[(23)] $FQ^0$
        \begin{enumerate}
            \item $B45$ (dual to $E12$)
            \begin{enumerate}
                \item $H(B45)= t(t-s)(1-t-s)a + t(st-ts)b$ 
                \item  $h(B45)= 1+t+ts -(1+t+t^2)= t(t-s)$
                \item Commuting $t$ and $s$ without additional assumptions; this is immediate from condition $(b)$.
                \item Solutions: Either $t=s$ or $t+s=1$ assuming no zero divisors.
                \item In the first when $t = s$, $c_0$ is arbitrary. In the second case when $t + s = 1$ (which we call the {\it Alexander solution}), $c_0$ must be $0$.
            \end{enumerate}

             \item $D24$ (self-dual)
            \begin{enumerate}
                \item  $H(D24)= (t-s)(1-t-s)a + (st-ts)tb + (st-ts)sc$
                \item   $h(D24)=(1+s+st)-(1+t+ts)=(s-t)+ (st-ts)$
                \item Commuting $t$ and $s$ without additional assumptions by condition $(b)$.
                \item Solutions: Either $t=s$ or $t+s=1$. 
                \item In the first when $t = s$, $c_0$ is arbitrary. In the second case when $t + s = 1$ (which we call the {\it Alexander solution}), $c_0$ must be $0$.
            \end{enumerate}

        \item $E12$ (dual to $B45$)
            \begin{enumerate}
                \item $H(E12) = s(s - t)(1 - s - t)b + s(st - ts)c$
                \item $h(E12) = (1 + s + s^2) - (1 + s + st) = s(s - t)$ 
                \item Commuting $t$ and $s$ without additional assumptions.
                \item Solutions: Either $t=s$ or $t+s=1$. 
                \item In the first when $t = s$, $c_0$ is arbitrary. In the second case when $t + s = 1$ (which we call the {\it Alexander solution}), $c_0$ must be $0$. 
            \end{enumerate}
        \end{enumerate}

    \item[(24)] $LNQ^L$
        \begin{enumerate}
            \item $A35$ (dual to $F13$)
            \begin{enumerate}
                \item    $H(A35) = t(t+s-t^2-ts)a+(st-ts)b+s(s-1)c$
                \item   $h(A35) = s-t^2$
                \item Commuting $t$ and $s$ without additional assumptions; see condition $(b)$.
                \item Solutions: $s^2=1$ and $t=1$. Assuming no zero divisors.  $(t,s) = (\pm 1, 1)$.
                \item No restrictions on $c_0$.
                \item Notice that if $(G, \cdot)$ is a loop, then $(G,\backslash)$ is an $LNQ^L$ quasigroup. Here, $a \backslash b = c$ means $a \cdot c = b$.
            \end{enumerate}
        \end{enumerate}    

    \item[(25)] $MNQ^2$
        \begin{enumerate}
            \item $C24$ (self-dual)
            \begin{enumerate}
                \item  $H(C24) = t(1-t)a+(st^2+sts -t^2s-ts^2)b+s(s-1)c$
                \item $h(C24) = s+st-t-ts$
                \item Commuting $t$ and $s$ without additional assumptions since $t=s=1$.
                \item Solutions: $t=s=1$ 
                \item No restrictions on $c_0$.
            \end{enumerate}
        \end{enumerate}

    \item[(26)] $RNQ^R$
        \begin{enumerate}
            \item $F13$ (dual to $A35$)
            \begin{enumerate}
                \item      $H(F13) = t(1-t)a +(st-ts)b+s((s-1)(t+s))c$
                \item  $h(F13)=s^2-t$
                \item Commuting $t$ and $s$ without additional assumptions; see condition $(b)$.
                \item Solutions: $t=1$ and $s^2=1$ Assuming no zero divisors, $(t,s) = (1, \pm 1)$ 
                \item No restrictions on $c_0$.
                \item Notice that if $(G, \cdot)$ is a loop, then $(G,/)$ is an $RNQ^R$ quasigroup. Here, $a/b =c$ means $c\cdot b = a$.
            \end{enumerate}
        \end{enumerate}  
        
    \end{enumerate}

\bigskip
\subsection{Homology, Boundary Maps}\label{S:homologyboundary}
Recall that a chain complex, $\mathcal C$, is a sequence of modules $C_n$ and homomorphisms $\partial_n: C_n \to C_{n-1}$ such that $\partial_{n-1}\partial_n=0$. The homology $H_n(\mathcal{C})$ is defined by 
$\frac{ker(\partial_n)}{im(\partial_{n+1})}$.

Any chain complex leads to a cochain complex by taking $C^n= \Hom(C_n\to k)$ and coboundary map 
$\partial^{n}: C^{n}\to C^{n+1} $ defined by $$\partial^{n}(f)(w)= f(\partial_n(w)),$$
where $f\in C^n$ and $w\in C_{n+1}$. Note that if $C_n$ are finitely generated free $k$-modules than the matrix of $\partial_n$ and $\partial^{n-1}$ are transpose one of the other and for $k$ PID the torsion parts are isomorphic ($tor(H_{n-1}(\mathcal C_\bullet))= tor(H^n(\mathcal C^{\bullet} ))$, in particular, $tor(H_{2})= tor(H^3)$).


\subsection{Boundary Map $\partial_3$}\label{S:boundarymaps}

To each Bol-Moufang relation $Vij$, we construct a boundary map $\partial^{Vij}_3:C_3 \to C_2$. As evidenced by the experimental data, $H_2(X;Vij) = \operatorname{Im}\partial_3^{Vij}/\ker \partial_2 \cong \operatorname{Im}\partial_3^{Wk\ell}/\ker \partial_2 = H_2(X;Wk\ell)$ if $Vij$ and $Wk\ell$ define the same variety of quasigroups; see Conjecture \ref{conj:identity-independence}. Consequently the list of boundary maps that follows is organized according to the 26 equational classes (varieties) of BMq's. Recall that identity $Vij$ yields the definition $\partial_3(x,y,z) = Q(Vi) - Q(Vj)$ where $Q$ is as given in Definition~\ref{def:polynomials}.

\begin{enumerate}
    \item[(2)] $RG1^L$ Quasigroups

    \begin{enumerate}
        \item $A25$
        \[ \partial_3(x,y,z)=st(x,y) + s(xy,z) + (x,(xy)z) - t^2 (x,x)  - t (xx,y)  - ((xx)y,z) \]
        \item $D25$ 
        \[\partial_3(x,y,z)=st(y,z) + s(yz,x) + (x,(yz)x) - t^2(x,y)  - t (xy,z)  - ((xy)z,x) \]
    \end{enumerate}
    
\item[(3)] $LG1^R$ Quasigroups

    \begin{enumerate}
        \item $F14$
        \[\partial_3(x,y,z)=s^2(z, z) + s(y, (z z)) + (x, y (z z)) - st (y, z)     -
  t(x, (y z))   -(x (y z), z)  \]
        \item $D14$
        \[ \partial_3(x,y,z)=s^2 (z, x) + s(y, z x) + (x, y (z x)) - st(y, z)  -
  t (x, y z)     -(x (y z), x)    \]
    \end{enumerate}

\item[(4)] $RG2^L$ Quasigroups

    \begin{enumerate}
        \item $A23$
        \[ \partial_3(x,y,z)=st (x, y)  + s(x y, z) + (x, (x y) z)  - s(y, z)  - 
 t(x, x)  -(x x, y z) \]
    \end{enumerate}

\item[(5)] $LG2^R$ Quasigroups

    \begin{enumerate}
        \item $F34$
        \[ \partial_3(x,y,z)= t(x, y) + s(z, z) + (x y, z z)- st (y, z)  - 
 t(x, y z)   - (x (y z), z)     \]
    \end{enumerate}

\item[(6)] $RG3^L$ Quasigroups

    \begin{enumerate}
        \item $B25$
                \[ \partial_3(x,y,z)= st (y, x)+ s(y x, z)  + (x, (y x) z)-t^2 (x, y)  - 
 (x y, x) t  - ((x y) x, z)  \]
    \end{enumerate}

\item[(7)] $LG3^R$ Quasigroups

    \begin{enumerate}
        \item $E14$
 \[ \partial_3(x,y,z)= s^2(z, y) + s(y, z y)+ (x, y (z y)) - st (y, z)  -
  t(x, y z)     - (x (y z), y)   \]
    \end{enumerate}

\item[(8)] $EQ^2$ Quasigroups:
            
    \begin{enumerate}
        \item $D15$
        \[ \partial_3(x,y,z)=  s^2 (z, x)+ s(y, z x) + (x, y (z x)) - t^2 (x, y)    -
  t(x y, z)  -((x y) z, x)   \]
        \item $B23$
        \[ \partial_3(x,y,z)=st (y, x)+ s (y x, z) + (x, (y x) z) - 
 t (x, y)  - s(x, z)   -(x y, x z)     \]
        \item $E34$
        \[ \partial_3(x,y,z)= t(x, y) + s (z, y)  + (x y, z y) - st (y, z) - 
 t (x, y z)    - (x (y z), y)   \]
    \end{enumerate}

\item[(9)]  $MQ^2$ Quasigroups :

    \begin{enumerate}
        \item $D34$
        \[ \partial_3(x,y,z)=  t(x, y) + s (z, x) + (x y, z x) - st (y, z) - 
 t (x, y z)   - (x (y z), x)     \]
        \item $B15$
        \[ \partial_3(x,y,z)= s^2 (x, z) + s(y, x z) + (x, y (x z))   - t^2 (x, y)  -
  t (x y, x)  - ((x y) x, z)   \]
        \item $D23$
        \[ \partial_3(x,y,z)=st (y, z) + s(y z, x)  + (x, (y z) x) - s (z, x)   -  t(x, y)  -(x y, z x) \]
        \item $E15$
        \[ \partial_3(x,y,z)= s^2 (z, y) + s (y, z y) + (x, y (z y)) - t^2 (x, y)   -
  t (x y, z)   -((x y) z, y)     \]
    \end{enumerate}

\item[(10)] Left Bol Quasigroups ($LBQ^R$):
    
    \begin{enumerate}
        \item $B14$
        \[ \partial_3(x,y,z)= s^2 (x, z) + s (y, x z) + (x, y (x z)) - st (y, x)  - t (x, y x)   - (x (y x), z)     \]
    \end{enumerate}

    \item[(11)] Right Bol Quasigroups ($RBQ^L$):

    \begin{enumerate}
        \item $E25$
        \[ \partial_3(x,y,z)= st (y, z) + s (y z, y) + (x, (y z) y) - t^2 (x, y)   -  t (x y, z)    - ((x y) z, y) \]
    \end{enumerate}

\item[(12)] C Quasigroups ($CQ^0$):

    \begin{enumerate}
        \item $C15$
        \[ \partial_3(x,y,z)=s^2 (y, z)  + s (y, y z) + (x, y (y z)) - t^2 (x, y)  -   t (x y, y)    - ((x y) y, z)  \]
    \end{enumerate}

\item[(13)] $LC1^2$ Quasigroups:

    \begin{enumerate}
        \item $A34$
        \[ \partial_3(x,y,z)= t(x, x) + s(y, z) + (x x, y z) - st(x, y) - 
 (x, x y) t   -(x (x y), z)    \]
    \end{enumerate}

\item[(14)] $LC2^0$ Quasigroups:

    \begin{enumerate}
        \item $A14$
        \[ \partial_3(x,y,z)=s^2  (y, z) + s(x, y z) + (x, x (y z)) - st (x, y)  -
  t (x, x y)  - (x (x y), z)   \]
    \end{enumerate}

\item[(15)] $LC3^L$ Quasigroups:

    \begin{enumerate}
        \item $A15$
        \[ \partial_3(x,y,z)= s^2 (y, z)  + s (x, y z) + (x, x (y z)) - t^2 (x, x)  -
  t (x x, y)  - ((x x) y, z)   \]
    \end{enumerate}

\item[(16)] $LC4^R$ Quasigroups:

    \begin{enumerate}
        \item $C14$
        \[ \partial_3(x,y,z)= s^2 (y, z) + s (y, y z) + (x, y (y z)) - st (y, y)- t (x, y y)  - (x (y y), z)     \]
    \end{enumerate}

\item[(17)] $RC1^2$ Quasigroups:

    \begin{enumerate}
        \item $F23$
        \[ \partial_3(x,y,z)=  st (y, z) + s (y z, z) + (x, (y z) z)  - t(x, y)  - s (z, z)  - (x y, z z).  \]
    \end{enumerate}

\item[(18)] $RC2^0$ Quasigroups:

    \begin{enumerate}
        \item $F25$
        \[ \partial_3(x,y,z)= st (y, z) + s (y z, z) + (x, (y z) z) - t^2 (x, y)  -  t((x y), z)  -((x y) z, z)  \]
    \end{enumerate}

\item[(19)] $RC3^R$ Quasigroups:

    \begin{enumerate}
        \item $F15$
        \[ \partial_3(x,y,z)= s^2 (z, z) + s (y, z z) + (x, y (z z)) - t^2 (x, y)     -
  t (x y, z)    -((x y) z, z) \]
    \end{enumerate}

\item[(20)] $RC4^L$ Quasigroups:

    \begin{enumerate}
        \item $C25$
        \[ \partial_3(x,y,z)=st (y, y) + s (y y, z) + (x, (y y) z) - t^2 (x, y)   - 
 t (x y, y)  -((x y) y, z)    \]
    \end{enumerate}

\item[(21)] $LAQ^L$ Quasigroups:

    \begin{enumerate}
        \item $A13$
        \[ \partial_3(x,y,z)=  s^2 (y, z) + s (x, y z) + (x, x (y z)) - t(x, x)   - s (y, z)  -(x x, y z)    \]
        \item $A45$
        \[ \partial_3(x,y,z)= st (x, y) + 
 t (x, x y)  + (x (x y), z) - t^2 (x, x) - t(x x, y)    -((x x) y, z) \]
        \item $C12$
        \[ \partial_3(x,y,z)=  s^2 (y, z) + 
 s (y, y z) + (x, y (y z)) - st(y, y)  - s (y y, z)     - (x, (y y) z)  \]
    \end{enumerate}

\item[(22)] $RAQ^R$ Quasigroups:

    \begin{enumerate}
        \item $C45$
        \[ \partial_3(x,y,z)=st (y, y)  + 
 t (x, y y) + (x (y y), z)  - t^2 (x, y) - t(x y, y)    -((x y) y, z) \]
        \item $F12$
        \[ \partial_3(x,y,z)=s^2 (z, z) + s (y, z z) + (x, y (z z)) - st (y, z)  - s (y z, z)     - (x, (y z) z) \]
        \item $F35$
        \[ \partial_3(x,y,z)= s (z, z) + t (x, y) + (x y, z z)- t^2 (x, y)    - 
 t (x y, z)    - ((x y) z, z)   \]
    \end{enumerate}

\item[(23)] $FQ^0$ Quasigroups:

    \begin{enumerate}
        \item $B45$
        \[ \partial_3(x,y,z)=  st (y, x) + 
 t (x, y x) + (x (y x), z) - t^2 (x, y)   - t(x y, x)   -((x y) x, z)    \]
        \item $D24$
     \[ \partial_3(x,y,z)=  s (y z, x)  + (x, (y z) x) - t (x, y z)   -(x (y z), x) \]
        \item $E12$
        \[ \partial_3(x,y,z)= s^2 (z, y) + 
 s (y, z y) + (x, y (z y)) - st (y, z) - s (y z, y)     - (x, (y z) y)  \]
    \end{enumerate}

\item[(24)] $LNQ^L$ Quasigroups:

    \begin{enumerate}
        \item $A35$
 
 \[ \partial_3(x,y,z)= t(x,x) +s (y,z) + (x x, y z) -t^2(x,x)  - 
 t (x x, y)  -((x x) y, z)
\]
    \end{enumerate}

\item[(25)] $MNQ^2$ Quasigroups:

    \begin{enumerate}
        \item $C24$
        \[ \partial_3(x,y,z)=  s (y y, z) + (x, (y y) z)  - t (x, y y)    -(x (y y), z) \]
    \end{enumerate}

\item[(26)] $RNQ^R$ Quasigroups:

    \begin{enumerate}
        \item $F13$
        \[ \partial_3(x,y,z)= s^2 (z, z) + s (y, z z)  + (x, y (z z)) - t (x, y)   - s (z, z)    - (x y, z z)  \]
    \end{enumerate}
\end{enumerate}

\begin{remark}\label{rem:id}
    In addition to the quasigroup varieties defined by the identities $Vij$ we can also consider the identity $Xij$ defined from the same rooted binary trees of shapes $i$ and $j$, but using four distinct labels instead of one repeated label. For example, the identity $X14$ is $x_1(x_2(x_3x_4)) = (x_1(x_2x_3))x_4$ obtained from the trees in Figures \ref{F:bracket_1} and \ref{F:bracket_4}. 

    In the same way we have obtained boundary maps $\partial_2$ and $\partial_3$ previously from identities, we obtain from the $X14$ identity $\partial_2(x,y) = tx + sy - xy$ and \[ \partial_3(x,y,z,w) = s^2(z,w) + s(y,zw) + (x, y(zw)) - ts(y,z) - t(x,yz) - (x(yz),w). \]

    Note that any of the identities $Vij$ are satisfied in a quasigroup satisfying $Xij$, and so the variety determined by $Xij$ is a subvariety of each of the varieties determined by the $Aij$ through $Fij$. This means that $Xij$ frequently determines the variety of {\it groups}. For instance, $X12$ implies $A12$ implies associativity, $X13$ implies $B13$ implies associativity, and so on. The only $Xij$ identities that don't determine the variety of groups are $X14$, its dual $X25$, and the self-dual $X15$.


    In the complexes we construct to define homology, the map $f$ taking $(x,y,z)$ to $(x,x,y,z)$ gives a chain map \[ \begin{tikzcd}
        0 \arrow[r] & (C_3;Aij) \arrow[r, "\partial_3"] \arrow[d, "f"] & (C_2;Aij) \arrow[r, "\partial_2"] \arrow[d, "1"] & (C_1;Aij) \arrow[r, "0"] \arrow[d, "1"] & 0 \\
        0 \arrow[r] & (C_3;Xij) \arrow[r, "\partial_3"] & (C_2;Xij) \arrow[r, "\partial_2"] & (C_1;Xij) \arrow[r, "0"] & 0 
    \end{tikzcd} \] inducing a map $1_*: H_2(\blank; Aij) \to H_2(\blank; Xij)$. This map is surjective, so $H_2$ computed with the $\partial_3$ we get from identity X$ij$ is always a quotient of $H_2$ computed with $\partial_3$ coming from identity $Aij$, and similarly for any other identity $Vij$ with a repeated letter. Hence $Xij$ homology is always a common quotient of homology computed with the $Vij$ identities. 
\end{remark}

\section{BMq Chain Complexes and Their Properties}\label{S:chain}

\subsection{Chain complex verification}\label{partialsquared}
In this subsection, we verify that 
\[ 0 \to C_3 \xrightarrow{\partial_3} C_2 \xrightarrow{\partial_2} C_1 \to 0 \]
 is, in fact, a chain complex. We perform the verification of $\partial_2 \partial_3 = 0$ in slightly greater generality than our situation demands.

 Recall from Definition~\ref{def:polynomials} the polynomial $Q$ on a binary rooted tree $T$:

\[ Q(T) = \sum_{v \in V(T) - L(T)} h_v(t,s) g(v) \]

where $g(v) = (g(v_L), g(v_R))$ with 
\[
\bar g(v)=
\left\{
\begin{array}{ll}
x_i & \mbox{when $v\in L(T)$ is decorated by $x_i$},\\
\bar g(v_L)\bar g(v_R) & \mbox{when $v\in V(T)-L(T)$ and $v_L$, $v_R$} \\
& \mbox{are the left/right children of $v$ in $T$}.
\end{array}
\right.
\]
 
This map is a generalization of $\partial_3: C_3 \to C_2$ in that 
$\partial_3(x,y,z) = Q(T_1) - Q(T_2)$ for two trees $T_1, T_2$ with $4$ leaves decorated by the same 4-letter word in $x,y,z$. 

Finally, recall that $i(v), j(v)$ are the number of left (resp. right) turns from the root vertex $r$ to the vertex $v.$ 

In general, we need not assume that $t,s$ commute, but since their commutativity follows in every case \footnote{Commutativity of $t,s$ is immediate for all but 4 cases in which we must assume no zero divisors.}  as in List \ref{list}, we work here with the assumption that they commute for simplicity. 

\begin{lemma}
    For  $r \in V$ the root vertex,
    \[ (\partial_2 \circ Q)(T) = - \bar{g}(r) + \sum_{l \in L(T)} t^{i(l)} s^{j(l)}\bar{g}(l) \]
\end{lemma}

\begin{proof}
    For the sake of clarity, we note that $\bar g (v) \in X$, $g(v) \in X^2$, and $\bar g = x_i$ for some $x_i \in W$.

By definition, 

\begin{align*}
    (\partial_2 \circ Q)(T) & = \sum_{v \in V(T)-L(T)} t^{i(v)}s^{j(v)} \partial_2(g(v)) \\
    & = \sum_{v \in V(T)-L(T)} t^{i(v)}s^{j(v)} (t \bar g(v_L) + s \bar g (v_R) - \bar g(v_L) * \bar g(v_R)) \\
     & = \sum_{v \in V(T)-L(T)} t^{i(v) +1}s^{j(v)} \bar g(v_L) + t^{i(v)}s^{j(v)+1} \bar g (v_R) - t^{i(v)}s^{j(v)}  \bar g(v) \\
\end{align*}

Assume for $v \in V(T)- L(T)$ that $v_L \in V_L$; i.e. the left child of $v$ is also a degree three vertex.

Then 
\[ \partial_2(g(v)) = {\color{blue} t^{i(v) +1}s^{j(v)} \bar g(v_L)} + {\color{red} t^{i(v)}s^{j(v)+1} \bar g (v_R) } - t^{i(v)}s^{j(v)}  \bar g(v)\]

 and 
 
\[ \partial_2(g(v_L)) = t^{i(v) +2}s^{j(v)} \bar g(v_{LL}) + t^{i(v)+1}s^{j(v)+1} \bar g (v_{LR}) {\color{blue} - t^{i(v)+1}s^{j(v)}  \bar g(v_L) } \]
 have a pair of cancelling terms in blue. Similarly the term in red will cancel with the final term in $\partial_2(g(v_R))$ so long as $v_R \in V(T)- L(T)$.

If a vertex $v \in V(T)-L(T)$ has a child which is not leaf, say $g(v_L) = x_i$, we produce a non-cancelling term  $t^{i(v_L)} s^{j(v_L)} \bar g(v_L)$. 

Finally, we note that the final term in $\partial_2(g(r))$ (namely, $-\bar{g}(r)$), will not cancel. Thus we arrive at the statement.
\end{proof}

With this lemma, we can quickly conclude,

\begin{corollary}

    For $T_1, T_2$ two trees with respective roots $r_1,r_2$ and the $\ell = |L(T_1)| = |L(T_2)|$ leaves decorated by the same elements of a quasigroup $X$, 
    \[ \partial_2( Q(T_1) - Q(T_2) ) = - \bar g(r_1) + \bar g(r_2) + \sum_{i = 1}^\ell (t^{\epsilon_1} s^{\delta_1} - t^{\epsilon_2} s^{\delta_2}) x_i\] for $\epsilon_i, \delta_j \in \mathbb{N}$.
\end{corollary}

In the specific case of $T_1,T_2$ 
with four leaves each decorated by $x,y,z$, $-\bar g (r_1) + \bar g (r_2)$ is precisely the Bol-Moufang Relation corresponding to the choice of trees and thus is assumed to be zero. Further, for $s=t=1,$ the sum $\sum_{i = 1}^{\ell} (t^{\epsilon_1} s^{\delta_1} - t^{\epsilon_2} s^{\delta_2}) x_i$ is clearly zero; in Section \ref{sec:substitutions}, we have found the full set of solutions $(t,s)$ for which this sum is zero by analyzing the structure of the affine BM quasigroup with the same relation.  Thus, $\partial_2(Q(T_1) - Q(T_2)) = \partial_2 \circ \partial_3 = 0$ in all cases where $(t,s)$ are solutions for the affine case with a given 
equation.  More precisely, we observe that $\partial_2(Q(T))\stackrel{x_i\mapsto a_i}{=}H(T)-rel(T)$ thus $\partial_2(Q(T_1)) - \partial_2(Q(T_2))\stackrel{x_i \mapsto a_i}{=} H(T_1)-H(T_2)\equiv 0$ by our analysis of affine quasigroups; in particular, this is always zero for $t=s=1$.

\subsection{First homology as abelianization}\label{S:H1Ab}

In group homology, we have that $H_1(G)$ is the abelianization $G^{ab}$. For any Bol-Moufang quasigroup $X$, the substitution $t = s = 1$ is always available and $H_1(X;1,1)$ has a similar interpretation as an abelianization in that any quasigroup homomorphism $f: X \to A$ for an abelian group $A$ factors through a unique group homomorphism $H_1(X;1,1) \to A$:

\begin{proposition}
    Let $X$ be a quasigroup. Then the functor $F$ taking $X$ to $\Z X/(a + b - ab)$ is left adjoint to the forgetful functor $U$ taking an abelian group to its underlying quasigroup. 
\end{proposition}

\begin{proof}
    We have a map $i_X : X \to \Z X / (a + b - ab)$ given by $x \mapsto [x]$, and for any $f: \Z X / (a + b - ab) \to A$ for an abelian group $A$, the composition $fi_X$ is a quasigroup homomorphism: $fi_X(xy) = f([xy]) = f([x] + [y]) = fi_X(x) + fi_X(y)$, so we have the map $i_X^*: \Hom(FX,A) \to \Hom(X,UA)$. Now if $g: X \to UA$ is a quasigroup map with $A$ an abelian group, then there is a unique group homomorphism $\varphi: \Z X \to A$ with $g = \varphi i_X$ since $\Z X$ is the free abelian group on $X$ (that is, $\varphi$ is $g$ extended by linearity and $\varphi \iota_X = g$ for $\iota_X$ the inclusion $X \hookrightarrow \Z X$). Therefore $\varphi(a + b - ab) = g(a) + g(b) - g(ab) = 0 \in A$ because $g$ is a quasigroup homomorphism, and so $\varphi$ induces a unique map on the quotient, say $\tilde{\varphi}: FX \to A$ with $\varphi = \tilde{\varphi} \pi$ for $\pi: \Z X \to FX$ the projection. This means $\tilde{\varphi}$ is the unique map with $g = \tilde{\varphi} \pi \iota_X = \tilde{\varphi} i_X$. That is, for every $g \in \Hom(X, UA)$, there is a unique $\tilde{\varphi} \in \Hom(FX, A)$ with $g = i_X^*(\tilde{\varphi})$, and $i_X^*$ is therefore a bijection. 

    For $\alpha: X \to X'$, $F\alpha: FX \to FX'$ is $\alpha$ extended by linearity, meaning $F\alpha$ is the unique group homomorphism with $F\alpha i_X = i_X' \alpha$, and this gives naturality in $X$. For $\beta: A \to B$, $\beta_*$ commutes with $i_X^*$, which gives naturality in $A$. Hence the bijection $\Hom(FX,A) \to \Hom(X, UA)$ is natural. 
\end{proof}

Hence we are justified in writing $H_1(X;1,1) = X^{ab}$. When available, we can interpret the substitutions $t = 1$, $s = -1$ and $t = -1$, $s = 1$ similarly as abelianizations of certain related quasigroups called the {\it parastrophes} or {\it conjugates}. These are the quasigroups with underlying set $X$ and operations \begin{align*}
    x / y &= z \iff z \cdot y = x \\
    x \backslash y &= z \iff x \cdot z = y \\
    x \circ y &= z \iff y \cdot x = z \\
    x \dslash y &= z \iff z \cdot x = y \\
    x \dbslash y &= z \iff y \cdot z = x.
\end{align*} Writing $(X,\cdot)$ as an abbreviation for $(X, \cdot, /, \backslash)$, we similarly write, e.g., $(X, /)$ as an abreviation for the parastrophe $(X, /, \cdot, \dbslash)$ and $(X, \backslash)$ for $(X, \backslash, \dslash, \cdot)$. 

 More information on parastrophes of quasigroups can be found in \cite[Section II.2]{Pfl2}, \cite[Section 1.2.2]{Shch} or \cite[Section 1.3]{Smi2}.

\begin{proposition}\label{prop:parastrophes}
    For $X$ a quasigroup, $H_1(X;1,-1) \cong (X,/)^{ab}$ and similarly, $H_1(X;-1,1) \cong (X,\backslash)^{ab}$. 
\end{proposition}

\begin{proof}
    We show that $H_1(X, \cdot;1,-1) \cong H_1(X,/;1,1) = (X,/)^{ab}$. We claim that the identity map on $\Z X$ induces an isomorphism \[ H_1(X,\cdot;1,-1) = \frac{\Z X}{a - b - a\cdot b} \stackrel{1_{\Z X}}{\longrightarrow} \frac{\Z X}{a + b - a/b} = H_1(X,/;1,1), \] and for this we check that the $a + b - a/b$ relations are consequences of the $a - b - a\cdot b$ relations, and vice versa. In $\Z X$, we have that $(a + b) - b = a$, and also that $(a/b) \cdot b = a$. Therefore in $\Z X / (a - b - a \cdot b)$, \[ (a + b) - b = a = (a/b) \cdot b = (a/b) - b \] implying that $a + b = a/b$. Similarly in $\Z X / (a + b - a/b)$ we have \[ (a - b) + b = a = (a \cdot b) / b = (a \cdot b) + b \] so that $a - b = a\cdot b$. Hence the identity map on $\Z X$ descends to well-defined maps on the quotients and $H_1(X, \cdot; 1, -1) \cong H_1(X, /; 1,1) = (X,/)^{ab}$. The argument that $H_1(X, \cdot; -1,1) \cong (X, \backslash)^{ab}$ is similar. 
\end{proof}

Note also that $\Z X/(a + b - ba) = \Z X/(b + a - ba)$, so $(X,\circ)^{ab} = (X,\cdot)^{ab}$; similarly $(X,\dslash)^{ab} = (X,/)^{ab}$ and $(X,\dbslash)^{ab} = (X,\backslash)^{ab}$. This says that $H_1$ for any of the parastrophes of $X$ is $H_1(X;t,s)$ with appropriate $(t,s)$ substitution.

It is not true that $H_1(X;t,s)$ is always the abelianization of some parastrophe, even if $t$ and $s$ are required to be invertible, for this leaves the case of $t = s = -1$, a substitution that is available in any of the CQ, LC2, RC2, or FQ varieties. 

\section{Examples}\label{S:examples}

The following is a collection of homology groups computed for the Bol-Moufang Quasigroups presented in \cite{PhVo1, PhVo2}. 

Given that $\partial_2(x,y) := t x + sy - xy$ and $\partial_1 = 0$, it is clear that the first homology for a quasigroup $X$ depends only on the choice of $t$ and $s$. Thus we will denote $H_1(X;t,s) = \ker \partial_1/\operatorname{Im}\partial_2$ for the given substitution of $t,s$. Each quasigroup below satisfies one or more non-equivalent  defining relations of Bol-Moufang type:  $Vij$ for $V \in \{A,B,C,D,E,F\}$ and $i,j \in \{1,2,3,4,5\}$. Each relation $Vij$ comes equipped with a set of solutions for $t,s$, presented in List  \ref{list}. Thus, for each quasigroup $X$, there will be as many first homology group computations as there are solutions $(t,s)$ in the union of all solutions sets for the satisfied relations.

In the case of the second homology, $\partial_3$ is defined with respect to a given Bol-Moufang relation. Thus we denote $H_2(X;Vij,t,s) = \ker \partial_2/\operatorname{Im}\partial_3^{Vij}$ where $\partial_3^{Vij}$ is the boundary map defined with respect to relation $Vij.$ Consistent with Conjecture~\ref{conj:identity-independence}, our data supports  $H_2(X,Vij,t,s) = H_2(X,Wk\ell,t,s)$ if $Vij$ and $Wk\ell$  define the same variety of Bol-Moufang quasigroup. Thus, we show only one computation $H_2(X,Vij,t,s)$ for $Vij$ a given representative of a variety.

\begin{example}

The following quasigroup satisfies $A25$, $A23$, $B25$, $E25$, $A14$, $F25$, $C25$, $A35$:
   \[ A_1 = \left(
\begin{array}{cccc}
 0 & 1 & 2 & 3 \\
 2 & 0 & 3 & 1 \\
 1 & 3 & 0 & 2 \\
 3 & 2 & 1 & 0 \\
\end{array}
\right) \]

\bigskip

$H_1(A_1;t,s):$\footnote{Note that this quasigroup is isomorphic to the quasigroup $(\Z/4\Z, \backslash)$ via the isomorphism swapping 2 and 3, i.e. the $ij$ entry of the table is $j - i$ mod 4, with the transposition (2\, 3) applied. Hence Proposition~\ref{prop:parastrophes} implies its first homology with $(t,s) = (-1,1)$ is $H_1(\Z/4\Z;1,1) = \Z/4\Z$, the usual group homology.}

\begin{center}
    
\begin{tabular}{c|c}
    $(t,s)$ & $H_1(A_1;t,s)$  \\
    \hline
    $(1,1)$ & $\Z/2\Z$ \\
    $(1,-1)$ & $\Z/4\Z$ \\
    $(-1,1)$ & $\Z/4\Z$ \\
    $(-1,-1)$ & $\Z/6\Z$ \\
\end{tabular}
\end{center}
\bigskip

$H_2(A_1;Vij,t,s):$

\begin{center}
    
\begin{tabular}{c|c|c}
    $Vij$ & $(t,s)$ & $H_2(A_1;t,s)$  \\
    \hline

    $A25$ & $(1,1)$ & $\Z/2\Z$ \\
    & $(-1,1)$ & $\Z/2\Z$ \\
    $A23$ & $(1,1)$ & $\Z/2\Z$ \\
     & $(-1,1)$ & $\Z/4\Z$ \\
        $B25$ & $(1,1)$ & $\Z \oplus \Z/4\Z$ \\
        & $(-1,1)$ & $(\Z/2\Z)^2$ \\
        $E25$ & $(1,1)$ & $\Z/2\Z$ \\
         & $(-1,1)$ & $\Z/2\Z$ \\
        $A14$ & $(1,1)$ & $\Z \oplus (\Z/2\Z)^2$ \\
         & $(1,-1)$ & $\Z \oplus (\Z/2\Z)^2$ \\  
         & $(-1,-1)$ & $\Z \oplus (\Z/2\Z)^2$ \\
        $F25$ & $(1,1)$ & $\Z^6 \oplus (\Z/2\Z)^2$ \\
         & $(-1,1)$ & $\Z^6 \oplus (\Z/2\Z)^2$ \\     
         & $(-1,-1)$ & $\Z^6 \oplus (\Z/2\Z)^2$ \\  
        $C25$ & $(1,1)$ & $\Z^3 \oplus (\Z/2\Z)^2$ \\
         & $(-1,1)$ & $\Z^3 \oplus (\Z/2\Z)^2$ \\
        $A35$ & $(1,1)$ & $\Z^9$ \\
         & $(-1,1)$ & $\Z^6 \oplus (\Z/2\Z)^3$ \\ 
\end{tabular}

\end{center}

\end{example}

\begin{example}
The following quasigroup satisfies 
$C15$, $A14$, $F25$, $B45$:
   \[ A_2 = \left(
\begin{array}{ccc}
 0 & 2 & 1 \\
 2 & 1 & 0 \\
 1 & 0 & 2 \\
\end{array}
\right) \]

\bigskip

$H_1(A_2;t,s):$

\begin{center}
    
\begin{tabular}{c|c}
    $(t,s)$ & $H_1(A_2;t,s)$  \\
    \hline
    $(1,1)$ & $0$ \\
    $(1,-1)$ & $0$ \\
    $(-1,1)$ & $0$ \\
    $(-1,-1)$ & $(\Z/3\Z)^2$ \\
\end{tabular}
\end{center}
\bigskip

$H_2(A_2;Vij,t,s):$

\begin{center}
    
\begin{tabular}{c|c|c}
    $Vij$ & $(t,s)$ & $H_2(A_2;t,s)$  \\
    \hline

    $C15$ & $(1,1)$ & $\Z$ \\
        & $(-1,-1)$ & $\Z$ \\
    $A14$ & $(1,1)$ & $\Z^3$ \\
        & $(1,-1)$ & $\Z^3$ \\
        & $(-1,-1)$ & $\Z^3$ \\
    $F25$ & $(1,1)$ & $\Z^3$ \\
        & $(-1,1)$ & $\Z^3$ \\
        & $(-1,-1)$ & $\Z^3$ \\
    $B45$ & $(1,1)$ & $\Z^3 \oplus \Z/2\Z$ \\
        & $(-1,-1)$ & $\Z^4$ \\
\end{tabular}

\end{center}

\end{example}

\begin{example}
The following quasigroup satisfies 
$F14$, $F34$, $E14$, $B14$, $A14$, $C14$, $F13$:
   \[ A_3 = \left( \begin{array}{ccc}
 0 & 2 & 1 \\
 1 & 0 & 2 \\
 2 & 1 & 0 \\
\end{array} \right)\]

\bigskip

$H_1(A_3;t,s):$\footnote{This quasigroup is $A_3 = (\Z/3\Z, /)$, i.e. $\Z/3Z$ with operation $xy = x - y$, so $H_1(A_3;1,-1) = \Z/3\Z$ (the usual group homology) again in agreement with Proposition~\ref{prop:parastrophes}.}

\begin{center}
    
\begin{tabular}{c|c}
    $(t,s)$ & $H_1(A_3;t,s)$  \\
    \hline
    $(1,1)$ & $0$ \\
    $(1,-1)$ & $\Z/3\Z$ \\
    $(-1,-1)$ & $\Z/3\Z$ \\
\end{tabular}
\end{center}
\bigskip

$H_2(A_3;Vij,t,s):$\footnote{This quasigroup also satisfies the X14 identity (note we list that it satisfies A14, B14, C14, E14, and F14, with D14 defining the same variety as F14), so we can compute its X14 homology $H_2 \cong 0$ to be the trivial group for $(t,s) = (1,1)$ and $(t,s) = (1,-1)$. This is consistent with Remark~\ref{rem:id}, as the only common quotient of $H_2(A_3;t,s)$ computed with respect to the various $Vij$ identities is the trivial group.}

\begin{center}
    
\begin{tabular}{c|c|c}
    $Vij$ & $(t,s)$ & $H_2(A_3;t,s)$  \\
    \hline

    $F14$ & $(1,1)$ & $0$ \\
        & $(1,-1)$ & $0$ \\
    $F34$ & $(1,1)$ & $0$ \\
        & $(1,-1)$ & $\Z^3$ \\
    $E14$ & $(1,1)$ & $0$ \\
        & $(1,-1)$ & $0$ \\
    $B14$ & $(1,1)$ & $0$ \\
        & $(1,-1)$ & $0$ \\
    $A14$ & $(1,1)$ & $\Z^3$ \\
        & $(1,-1)$ & $\Z^3$ \\
        & $(-1,-1)$ & $\Z^3$ \\
    $C14$ & $(1,1)$ & $\Z$ \\
        & $(1,-1)$ & $\Z$ \\
    $F13$ & $(1,1)$ & $\Z^4$ \\
        & $(1,-1)$ & $\Z^2 \oplus \Z/2\Z$ \\
\end{tabular}

\end{center}

\end{example}

\begin{example}
The following quasigroup satisfies 
$A23$, $A15$:
   \[ A_4 = \left( \begin{array}{ccccc}
 0 & 1 & 2 & 3 & 4 \\
 2 & 3 & 1 & 4 & 0 \\
 3 & 0 & 4 & 2 & 1 \\
 4 & 2 & 0 & 1 & 3 \\
 1 & 4 & 3 & 0 & 2 \\
\end{array} \right) \]

\bigskip

$H_1(A_4;t,s):$

\begin{center}
    
\begin{tabular}{c|c}
    $(t,s)$ & $H_1(A_4;t,s)$  \\
    \hline
    $(1,1)$ & $0$ \\
    $(-1,1)$ & $0$ \\
    $(-2,1)$ & $\Z/10\Z$ \\
\end{tabular}
\end{center}
\bigskip

$H_2(A_4;Vij,t,s):$

\begin{center}
    
\begin{tabular}{c|c|c}
    $Vij$ & $(t,s)$ & $H_2(A_4;t,s)$  \\
    \hline
    $A15$ & $(1,1)$ & $0$ \\
        & $(-2,1)$ & $0$ \\
    $A23$ & $(1,1)$ & $0$ \\
        & $(-1,1)$ & $0$ \\
\end{tabular}
\end{center}

\end{example}

\begin{example}
The following quasigroup satisfies 
$A13$:
   \[ A_5 =\left(
\begin{array}{cccccc}
 0 & 1 & 2 & 3 & 4 & 5 \\
 2 & 3 & 4 & 0 & 5 & 1 \\
 1 & 4 & 3 & 5 & 0 & 2 \\
 4 & 0 & 5 & 2 & 1 & 3 \\
 3 & 5 & 0 & 1 & 2 & 4 \\
 5 & 2 & 1 & 4 & 3 & 0 \\
\end{array}
\right)\]

\bigskip

$H_1(A_5;t,s):$

\begin{center}
    
\begin{tabular}{c|c}
    $(t,s)$ & $H_1(A_5;t,s)$  \\
    \hline
    $(1,1)$ & $\Z/3\Z$ \\
\end{tabular}
\end{center}
\bigskip

$H_2(A_5;Vij,t,s):$

\begin{center}
    
\begin{tabular}{c|c|c}
    $Vij$ & $(t,s)$ & $H_2(A_5;t,s)$  \\
    \hline

    $A13$ & $(1,1)$ & $\Z^5$ \\
\end{tabular}

\end{center}

\end{example}

\begin{example}
The following quasigroup satisfies 
$F14$, $F34$, $E14$, $B14$, $A14$, $C14$, $F13$:
   \[ A_6 = \left(
\begin{array}{cccccc}
 0 & 2 & 1 & 3 & 5 & 4 \\
 1 & 4 & 0 & 5 & 3 & 2 \\
 2 & 0 & 5 & 4 & 1 & 3 \\
 3 & 5 & 4 & 0 & 2 & 1 \\
 4 & 1 & 3 & 2 & 0 & 5 \\
 5 & 3 & 2 & 1 & 4 & 0 \\
\end{array}
\right) \]

\bigskip

$H_1(A_6;t,s):$

\begin{center}
    
\begin{tabular}{c|c}
    $(t,s)$ & $H_1(A_6;t,s)$  \\
    \hline
    $(1,1)$ & $\Z/2\Z$ \\
    $(1,-1)$ & $\Z/2\Z$ \\
    $(-1,-1)$ & $\Z/6\Z$ \\
\end{tabular}
\end{center}
\bigskip

$H_2(A_6;Vij,t,s):$\footnote{Again we have an example of a quasigroup satisfying the X14 identity (see Remark \ref{rem:id}), and this time its X14 homology is $H_2 \cong \Z/2\Z$ with $(t,s) = (1,1)$ and with $(t,s) = (1,-1)$. The same comment applies to Example 5.7, and in both cases we note that $\Z/2\Z$ is a common quotient of all other $H_2(A_6)$ groups computed with respect to the various $Vij$ identities.}

\begin{center}
    
\begin{tabular}{c|c|c}
    $Vij$ & $(t,s)$ & $H_2(A_6;t,s)$  \\
    \hline

    $F14$ & $(1,1)$ & $\Z/2\Z$ \\
        & $(1,-1)$ & $\Z/2\Z$ \\
    $E14$ & $(1,1)$ & $\Z \oplus \Z/3\Z$ \\
        & $(1,-1)$ & $\Z/2\Z$ \\
    $B14$ & $(1,1)$ & $\Z/6\Z$ \\
        & $(1,-1)$ & $\Z/6\Z$ \\
    $A14$ & $(1,1)$ & $\Z \oplus \Z/2\Z$ \\
        & $(1,-1)$ & $\Z \oplus \Z/2\Z$ \\
        & $(-1,-1)$ & $\Z \oplus \Z/2\Z$ \\
    $C14$ & $(1,1)$ & $\Z/2\Z$ \\
        & $(1,-1)$ & $\Z/2\Z$ \\

\end{tabular}

\end{center}

\end{example}

\begin{example}
The following quasigroup satisfies 
$F14$, $E14$, $B14$, $A14$, $C14$:
   \[ A_7 = \left(
\begin{array}{cccccc}
 0 & 3 & 4 & 1 & 2 & 5 \\
 1 & 2 & 5 & 0 & 3 & 4 \\
 2 & 1 & 0 & 5 & 4 & 3 \\
 3 & 0 & 1 & 4 & 5 & 2 \\
 4 & 5 & 2 & 3 & 0 & 1 \\
 5 & 4 & 3 & 2 & 1 & 0 \\
\end{array}
\right)\]

\bigskip

$H_1(A_7;t,s):$

\begin{center}
    
\begin{tabular}{c|c}
    $(t,s)$ & $H_1(A_7;t,s)$  \\
    \hline
    $(1,1)$ & $\Z/2\Z$ \\
    $(1,-1)$ & $\Z/2\Z$ \\
    $(-1,-1)$ & $\Z/6\Z$ \\
\end{tabular}
\end{center}
\bigskip

$H_2(A_3;Vij,t,s):$

\begin{center}
    
\begin{tabular}{c|c|c}
    $Vij$ & $(t,s)$ & $H_2(A_7;t,s)$  \\
    \hline

    $F14$ & $(1,1)$ & $\Z/2\Z$ \\
        & $(1,-1)$ & $\Z/2\Z$ \\
    $E14$ & $(1,1)$ & $\Z \oplus \Z/3\Z$ \\
        & $(1,-1)$ & $\Z/2\Z$ \\
    $B14$ & $(1,1)$ & $\Z/6\Z$ \\
        & $(1,-1)$ & $\Z/6\Z$ \\
    $A14$ & $(1,1)$ & $\Z \oplus \Z/2\Z$ \\
        & $(1,-1)$ & $\Z \oplus \Z/2\Z$ \\
        & $(-1,-1)$ & $\Z \oplus \Z/2\Z$ \\
    $C14$ & $(1,1)$ & $\Z/2\Z$ \\
        & $(1,-1)$ & $\Z/2\Z$ \\
\end{tabular}
\end{center}
\end{example}

\begin{example}
The following quasigroup satisfies 
$A14$, $C14$:
   \[ A_8 = \left( \begin{array}{cccccc}
    0 & 1 & 3 & 2 & 5 & 4 \\
    1 & 5 & 0 & 4 & 2 & 3 \\
    2 & 0 & 4 & 5 & 3 & 1 \\
    3 & 4 & 5 & 0 & 1 & 2 \\
    4 & 2 & 1 & 3 & 0 & 5 \\
    5 & 3 & 2 & 1 & 4 & 0 \\
\end{array} \right)\]

\bigskip

$H_1(A_8;t,s):$

\begin{center}
    
\begin{tabular}{c|c}
    $(t,s)$ & $H_1(A_8;t,s)$  \\
    \hline
    $(1,1)$ & $\Z/2\Z$ \\
    $(1,-1)$ & $\Z/2\Z$ \\
    $(-1,-1)$ & $\Z/6\Z$ \\
\end{tabular}
\end{center}
\bigskip

$H_2(A_8;Vij,t,s):$

\begin{center}
    
\begin{tabular}{c|c|c}
    $Vij$ & $(t,s)$ & $H_2(A_8;t,s)$  \\
    \hline

    $A14$ & $(1,1)$ & $\Z \oplus \Z/2\Z$ \\
        & $(1,-1)$ & $\Z \oplus \Z/2\Z$ \\
        & $(-1,-1)$ & $\Z \oplus \Z/2\Z$ \\
    $C14$ & $(1,1)$ & $\Z/2\Z$ \\
        & $(1,-1)$ & $\Z/2\Z$ \\
\end{tabular}

\end{center}

\end{example}

\begin{example}
The following quasigroup satisfies 
$E14$, $F25$, $F13$:
   \[ A_9 = \left(
\begin{array}{cccccc}
 0 & 2 & 3 & 1 & 4 & 5 \\
 1 & 0 & 4 & 5 & 2 & 3 \\
 2 & 4 & 5 & 0 & 3 & 1 \\
 3 & 5 & 0 & 4 & 1 & 2 \\
 4 & 3 & 1 & 2 & 5 & 0 \\
 5 & 1 & 2 & 3 & 0 & 4 \\
\end{array}
\right)\]

\bigskip

$H_1(A_9;t,s):$

\begin{center}
    
\begin{tabular}{c|c}
    $(t,s)$ & $H_1(A_9;t,s)$  \\
    \hline
    $(1,1)$ & $\Z/2\Z$ \\
    $(1,-1)$ & $\Z/2\Z$ \\
    $(-1,1)$ & $\Z/2\Z$ \\
    $(-1,-1)$ & $\Z/3\Z$ \\
\end{tabular}
\end{center}
\bigskip

$H_2(A_9;Vij,t,s):$

\begin{center}
    
\begin{tabular}{c|c|c}
    $Vij$ & $(t,s)$ & $H_2(A_9;t,s)$  \\
    \hline

    $E14$ & $(1,1)$ & $\Z$ \\
        & $(1,-1)$ & $\Z/2\Z$ \\
    $F25$ & $(1,1)$ & $\Z \oplus \Z/2\Z$ \\
        & $(-1,1)$ & $\Z \oplus \Z/6\Z$ \\
        & $(-1,-1)$ & $\Z \oplus \Z/2\Z$ \\
    $F13$ & $(1,1)$ & $\Z^5$ \\
        & $(1,-1)$ & $\Z/2\Z \oplus (\Z/2\Z)^3$ \\
\end{tabular}

\end{center}

\end{example}

\begin{example}
The following quasigroup satisfies 
$F34$, $F25$:
   \[ A_{10} = \left(
\begin{array}{cccc}
 0 & 2 & 3 & 1 \\
 1 & 3 & 2 & 0 \\
 2 & 0 & 1 & 3 \\
 3 & 1 & 0 & 2 \\
\end{array}
\right)\]

\bigskip

$H_1(A_{10};t,s):$

\begin{center}
    
\begin{tabular}{c|c}
    $(t,s)$ & $H_1(A_{10};t,s)$  \\
    \hline
    $(1,1)$ & $0$ \\
    $(1,-1)$ & $0$ \\
    $(-1,1)$ & $0$ \\
    $(-1,-1)$ & $\Z/3\Z$ \\
\end{tabular}
\end{center}
\bigskip

$H_2(A_{10};Vij,t,s):$

\begin{center}
    
\begin{tabular}{c|c|c}
    $Vij$ & $(t,s)$ & $H_2(A_{10};t,s)$  \\
    \hline

    $F34$ & $(1,1)$ & $0$ \\
        & $(1,-1)$ & $0$ \\
    $F25$ & $(1,1)$ & $\Z^6 \oplus (\Z/2\Z)^3$ \\
        & $(-1,1)$ & $\Z^6 \oplus (\Z/2\Z)^3$ \\
        & $(-1,-1)$ & $\Z^6 \oplus (\Z/2\Z)^3$ \\
\end{tabular}
\end{center}

\end{example}

\begin{example}
The following quasigroup satisfies 
$F34$, $F25$:
   \[ A_{11} = \left(
\begin{array}{cccc}
 0 & 3 & 1 & 2 \\
 1 & 2 & 0 & 3 \\
 2 & 1 & 3 & 0 \\
 3 & 0 & 2 & 1 \\
\end{array}
\right)\]

\bigskip

$H_1(A_{11};t,s):$

\begin{center}
    
\begin{tabular}{c|c}
    $(t,s)$ & $H_1(A_{11};t,s)$  \\
    \hline
    $(1,1)$ & $0$ \\
    $(1,-1)$ & $0$ \\
    $(-1,1)$ & $0$ \\
    $(-1,1)$ & $\Z/3\Z$ \\
\end{tabular}
\end{center}
\bigskip

$H_2(A_{11};Vij,t,s):$

\begin{center}
    
\begin{tabular}{c|c|c}
    $Vij$ & $(t,s)$ & $H_2(A_{11};t,s)$  \\
    \hline

    $F34$ & $(1,1)$ & $0$ \\
        & $(1,-1)$ & $0$ \\
    $F25$ & $(1,1)$ & $\Z^6 \oplus (\Z/2\Z)^3$ \\
          & $(-1,1)$ & $\Z^6 \oplus (\Z/2\Z)^3$ \\
          & $(-1,-1)$ & $\Z^6 \oplus (\Z/2\Z)^3$ \\
\end{tabular}
\end{center}
\end{example}

\begin{example}
The following quasigroup satisfies 
$A23$, $A15$:
   \[ A_{12} = \left(
\begin{array}{ccccc}
 0 & 1 & 2 & 3 & 4 \\
 3 & 2 & 4 & 1 & 0 \\
 4 & 3 & 1 & 0 & 2 \\
 2 & 0 & 3 & 4 & 1 \\
 1 & 4 & 0 & 2 & 3 \\
\end{array}
\right) \]

\bigskip

$H_1(A_{12};t,s):$

\begin{center}
    
\begin{tabular}{c|c}
    $(t,s)$ & $H_1(A_{12};t,s)$  \\
    \hline
    $(1,1)$ & $0$ \\
    $(-1,1)$ & $0$ \\
    $(-2,1)$ & $\Z/10\Z$ \\
\end{tabular}
\end{center}
\bigskip

$H_2(A_{12};Vij,t,s):$

\begin{center}
    
\begin{tabular}{c|c|c}
    $Vij$ & $(t,s)$ & $H_2(A_{12};t,s)$  \\
    \hline

    $A23$ & $(1,1)$ & $0$ \\
        & $(1,-1)$ & $0$ \\
    $A15$ & $(1,1)$ & $0$ \\
    & $(-2,1)$ & $0$ \\
\end{tabular}
\end{center}
\end{example}

\begin{example}
The following quasigroup satisfies 
$B45$, $A35$, $C24$, $F13$:
   \[ A_{13} = \left(
\begin{array}{ccccc}
 0 & 1 & 2 & 3 & 4 \\
 1 & 0 & 3 & 4 & 2 \\
 2 & 4 & 0 & 1 & 3 \\
 3 & 2 & 4 & 0 & 1 \\
 4 & 3 & 1 & 2 & 0 \\
\end{array}
\right)\]

\bigskip

$H_1(A_{13};t,s):$

\begin{center}
    
\begin{tabular}{c|c}
    $(t,s)$ & $H_1(A_{13};t,s)$  \\
    \hline
    $(1,1)$ & $0$ \\
    $(1,-1)$ & $0$ \\
    $(-1,1)$ & $0$ \\
    $(-1,-1)$ & $\Z/3\Z$ \\
\end{tabular}
\end{center}
\bigskip

$H_2(A_{13};Vij,t,s):$

\begin{center}
    
\begin{tabular}{c|c|c}
    $Vij$ & $(t,s)$ & $H_2(A_{13};t,s)$  \\
    \hline

    $B45$ & $(1,1)$ & $\Z^8$ \\
        & $(-1,-1)$ & $\Z^8$ \\
    $A35$ & $(1,1)$ & $\Z^{16}$ \\
        & $(-1,1)$ & $\Z^{12} \oplus (\Z/2\Z)^4$ \\
    $C24$ & $(1,1)$ & $\Z^{12}$ \\
    $F13$ & $(1,1)$ & $\Z^{16}$ \\
        & $(1,-1)$ & $\Z^{12} \oplus (\Z/2\Z)^4$ \\
\end{tabular}

\end{center}

\end{example}

\begin{example}
The following quasigroup satisfies 
$C15$:
   \[ A_{14} = \left(
\begin{array}{cccccccc}
 1 & 3 & 0 & 2 & 7 & 6 & 4 & 5 \\
 4 & 2 & 6 & 0 & 5 & 3 & 7 & 1 \\
 0 & 6 & 2 & 4 & 3 & 5 & 1 & 7 \\
 7 & 5 & 4 & 6 & 1 & 2 & 0 & 3 \\
 2 & 0 & 3 & 1 & 6 & 7 & 5 & 4 \\
 6 & 4 & 5 & 7 & 2 & 1 & 3 & 0 \\
 3 & 7 & 1 & 5 & 0 & 4 & 2 & 6 \\
 5 & 1 & 7 & 3 & 4 & 0 & 6 & 2 \\
\end{array}
\right)\]

\bigskip

$H_1(A_{14};t,s):$

\begin{center}
    
\begin{tabular}{c|c}
    $(t,s)$ & $H_1(A_{14};t,s)$  \\
    \hline
    $(1,1)$ & $\Z/4\Z$ \\
    $(-1,-1)$ & $\Z/6\Z$ \\
\end{tabular}
\end{center}
\bigskip

$H_2(A_{14};Vij,t,s):$

\begin{center}
    
\begin{tabular}{c|c|c}
    $Vij$ & $(t,s)$ & $H_2(A_{14};t,s)$  \\
    \hline

    $C15$ & $(1,1)$ & $\Z/2\Z$ \\
        & $(-1,-1)$ & $\Z/2\Z$ \\
\end{tabular}

\end{center}

\end{example}

\begin{example}
The following quasigroup satisfies 
$C15$, $A34$, $A14$, $A15$, $C14$, $F23$, $F25$, $F15$, $C25$, $A13$, $C45$, $A35$, $C24$, $F13$:
   \[ A_{15} = \left(
\begin{array}{cccccccccccc}
 0 & 1 & 2 & 3 & 4 & 5 & 6 & 7 & 8 & 9 & 10 & 11 \\
 1 & 2 & 0 & 4 & 5 & 3 & 7 & 8 & 6 & 10 & 11 & 9 \\
 2 & 0 & 1 & 5 & 3 & 4 & 8 & 6 & 7 & 11 & 9 & 10 \\
 3 & 4 & 5 & 0 & 1 & 2 & 10 & 11 & 9 & 8 & 6 & 7 \\
 4 & 5 & 3 & 1 & 2 & 0 & 11 & 9 & 10 & 6 & 7 & 8 \\
 5 & 3 & 4 & 2 & 0 & 1 & 9 & 10 & 11 & 7 & 8 & 6 \\
 6 & 7 & 8 & 11 & 9 & 10 & 0 & 1 & 2 & 4 & 5 & 3 \\
 7 & 8 & 6 & 9 & 10 & 11 & 1 & 2 & 0 & 5 & 3 & 4 \\
 8 & 6 & 7 & 10 & 11 & 9 & 2 & 0 & 1 & 3 & 4 & 5 \\
 9 & 10 & 11 & 7 & 8 & 6 & 5 & 3 & 4 & 0 & 1 & 2 \\
 10 & 11 & 9 & 8 & 6 & 7 & 3 & 4 & 5 & 1 & 2 & 0 \\
 11 & 9 & 10 & 6 & 7 & 8 & 4 & 5 & 3 & 2 & 0 & 1 \\
\end{array}
\right)\]

\bigskip

$H_1(A_{15};t,s):$

\begin{center}
    
\begin{tabular}{c|c}
    $(t,s)$ & $H_1(A_{15};t,s)$  \\
    \hline
    $(1,1)$ & $(\Z/2\Z)^2$ \\
    $(1,-1)$ & $(\Z/2\Z)^2$ \\
    $(1,-2)$ & $\Z/2\Z$ \\
    $(-1,1)$ & $(\Z/2\Z)^2$ \\
     $(-2,1)$ & $\Z/2\Z$ \\
     $(-1,-1)$ & $\Z/2\Z \oplus \Z/6\Z$
\end{tabular}
\end{center}
\bigskip

$H_2(A_{15};Vij,t,s):$

\begin{center}
    
\begin{tabular}{c|c|c}
    $Vij$ & $(t,s)$ & $H_2(A_{15};t,s)$  \\
    \hline

    $C15$ & $(1,1)$ & $\Z \oplus (\Z/2\Z)^3$ \\
        & $(-1,-1)$ & $\Z \oplus (\Z/2\Z)^3$ \\
    $A34$ & $(1,1)$ & $\Z^3$ \\
    $A14$ & $(1,1)$ & $\Z^6 \oplus (\Z/2\Z)^3$ \\
    & $(1,-1)$ & $\Z^6 \oplus (\Z/6\Z)^3$ \\
    & $(-1,-1)$ & $\Z^6 \oplus (\Z/2\Z)^3$ \\
    $A15$ & $(1,1)$ & $\Z^3 \oplus \Z/3\Z$ \\ 
    & $(-2,1)$ & $\Z^3 \oplus (\Z/2\Z)^2 \oplus \Z/6\Z$ \\
        
    $C14$ & $(1,1)$ & $\Z/3\Z \oplus (\Z/2\Z)^3$ \\
        & $(1,-1)$ & $\Z/3\Z \oplus (\Z/6\Z)^3$ \\
    $F23$ & $(1,1)$ & $\Z^3$ \\
    $F25$ & $(1,1)$ & $\Z^6 \oplus (\Z/2\Z)^3$ \\
        & $(-1,1)$ & $\Z^6 \oplus (\Z/6\Z)^3$ \\
            & $(-1,-1)$ & $\Z^6 \oplus (\Z/2\Z)^3$ \\
    $F15$ & $(1,1)$ & $\Z^3 \oplus (\Z/2\Z)^3$ \\
        & $(1,-2)$ & $\Z^3 \oplus (\Z/2\Z)^2 \oplus \Z/6\Z$ \\
    $C25$ & $(1,1)$ & $\Z^3 \oplus (\Z/2\Z)^3$ \\
        & $(-1,1)$ & $\Z^3 \oplus (\Z/2\Z)^3$ \\
    $A13$ & $(1,1)$ & $\Z^{28}$ \\
    $C45$ & $(1,1)$ & $\Z^{28}$ \\
    $A35$ & $(1,1)$ & $\Z^{33}$ \\
    & $(-1,1)$ & $\Z^{24} \oplus (\Z/2\Z)^9$ \\
    $C24$ & $(1,1)$ & $\Z^{30}$ \\
    $F13$ & $(1,1)$ & $\Z^{33}$ \\ 
    & $(1,-1)$ & $\Z^{24} \oplus (\Z/2\Z)^9$ \\
\end{tabular}
\end{center}

\end{example}

\begin{example}
The following quasigroup satisfies 
$B14$, $A34$, $A14$, $A15$, $C14$, $A13$, $A35$, $C24$, $F13$:
   \[ A_{16} = \left(
\begin{array}{cccccccc}
 0 & 1 & 2 & 3 & 4 & 5 & 6 & 7 \\
 1 & 0 & 3 & 2 & 5 & 4 & 7 & 6 \\
 2 & 3 & 0 & 1 & 6 & 7 & 4 & 5 \\
 3 & 5 & 1 & 7 & 0 & 6 & 2 & 4 \\
 4 & 2 & 6 & 0 & 7 & 1 & 5 & 3 \\
 5 & 4 & 7 & 6 & 1 & 0 & 3 & 2 \\
 6 & 7 & 4 & 5 & 2 & 3 & 0 & 1 \\
 7 & 6 & 5 & 4 & 3 & 2 & 1 & 0 \\
\end{array}
\right)\]

\bigskip

$H_1(A_{16};t,s):$

\begin{center}
    
\begin{tabular}{c|c}
    $(t,s)$ & $H_1(A_{16};t,s)$  \\
    \hline
    $(1,1)$ & $(\Z/2\Z)^2$ \\
    $(1,-1)$ & $(\Z/2\Z)^2$ \\
    $(-1,1)$ & $(\Z/2\Z)^2$ \\
    $(-1,-1)$ & $\Z/2\Z \oplus \Z/6\Z$ \\
    $(-2,1)$ & $\Z/2\Z$ \\
\end{tabular}
\end{center}
\bigskip

$H_2(A_{16};Vij,t,s):$

\begin{center}
    
\begin{tabular}{c|c|c}
    $Vij$ & $(t,s)$ & $H_2(A_{16};t,s)$  \\
    \hline

    $B14$ & $(1,1)$ & $(\Z/2\Z)^7$ \\
        & $(1,-1)$ & $(\Z/2\Z)^7$ \\
    $A34$ & $(1,1)$ & $\Z^3 \oplus (\Z/2\Z)^2$ \\
    $A14$ & $(1,1)$ & $\Z^6 \oplus (\Z/2\Z)^6$ \\
        & $(1,-1)$ & $\Z^6 \oplus (\Z/2\Z)^6$  \\
        & $(-1,-1)$ & $\Z^6 \oplus (\Z/2\Z)^6$  \\
    $A15$ & $(1,1)$ & $\Z^3 \oplus (\Z/2\Z)^2$ \\
    & $(-2,1)$ & $\Z^3 \oplus (\Z/2\Z)^3$ \\ 
    $C14$ & $(1,1)$ & $\Z/3\Z \oplus (\Z/2\Z)^6$ \\
        & $(1,-1)$ & $\Z/3\Z \oplus (\Z/2\Z)^6$ \\
    $A13$ & $(1,1)$ & $\Z^{17}$ \\
    $A35$ & $(1,1)$ & $\Z^{21}$ \\
        & $(-1,1)$ & $\Z^{15} \oplus (\Z/2\Z)^6$ \\
    $C24$ & $(1,1)$ & $\Z^{18}$ \\
    $F13$ & $(1,1)$ & $\Z^{21}$ \\
    & $(1,-1)$ & $\Z^{15} \oplus (\Z/2\Z)^6$ \\
\end{tabular}

\end{center}

\end{example}

\begin{example}
The following quasigroup satisfies 
$B14$, $A34$, $A14$, $A15$, $C14$, $A13$, $A35$, $C24$:
   \[ A_{17} = \left(
\begin{array}{cccccccccccc}
 0 & 1 & 2 & 3 & 4 & 5 & 6 & 7 & 8 & 9 & 10 & 11 \\
 1 & 0 & 3 & 2 & 5 & 4 & 7 & 6 & 9 & 8 & 11 & 10 \\
 2 & 3 & 0 & 1 & 6 & 7 & 4 & 5 & 10 & 11 & 8 & 9 \\
 3 & 2 & 6 & 7 & 0 & 1 & 10 & 11 & 4 & 5 & 9 & 8 \\
 4 & 5 & 1 & 0 & 8 & 9 & 2 & 3 & 11 & 10 & 6 & 7 \\
 5 & 4 & 8 & 9 & 1 & 0 & 11 & 10 & 2 & 3 & 7 & 6 \\
 6 & 7 & 11 & 10 & 2 & 3 & 8 & 9 & 1 & 0 & 4 & 5 \\
 7 & 6 & 10 & 11 & 3 & 2 & 9 & 8 & 0 & 1 & 5 & 4 \\
 8 & 9 & 5 & 4 & 11 & 10 & 1 & 0 & 7 & 6 & 2 & 3 \\
 9 & 8 & 4 & 5 & 10 & 11 & 0 & 1 & 6 & 7 & 3 & 2 \\
 10 & 11 & 7 & 6 & 9 & 8 & 3 & 2 & 5 & 4 & 0 & 1 \\
 11 & 10 & 9 & 8 & 7 & 6 & 5 & 4 & 3 & 2 & 1 & 0 \\
\end{array}
\right)\]

\bigskip

$H_1(A_{17};t,s):$

\begin{center}
    
\begin{tabular}{c|c}
    $(t,s)$ & $H_1(A_{17};t,s)$  \\
    \hline
    $(1,1)$ & $(\Z/2\Z)^2$ \\
    $(1,-1)$ & $(\Z/2\Z)^2$ \\
    $(-1,1)$ & $(\Z/2\Z)^2$ \\
    $(-1,-1)$ & $\Z/2\Z \oplus \Z/6\Z$ \\
    $(-2,1)$ & $\Z/2\Z$ \\
\end{tabular}
\end{center}
\bigskip

$H_2(A_{17};Vij,t,s):$

\begin{center}
    
\begin{tabular}{c|c|c}
    $Vij$ & $(t,s)$ & $H_2(A_{17};t,s)$  \\
    \hline

    $B14$ & $(1,1)$ & $(\Z/2\Z)^4 \oplus \Z/6 \Z$ \\
        & $(1,-1)$ & $(\Z/2\Z)^5$ \\
    $A34$ & $(1,1)$ & $\Z^3$ \\
    $A14$ & $(1,1)$ & $\Z^6 \oplus (\Z/2\Z)^3$ \\
        & $(1,-1)$ & $\Z^6 \oplus (\Z/2\Z) \oplus (\Z/6\Z)^2$ \\
        & $(-1,-1)$ & $\Z^6 \oplus (\Z/2\Z)^3$ \\
    $A15$ & $(1,1)$ & $\Z^3$ \\
        & $(-2,1)$ & $\Z^3 \oplus (\Z/2\Z)^3$ \\
    $C14$ & $(1,1)$ & $\Z^3 \oplus (\Z/2\Z)^3$ \\
        & $(1,-1)$ & $\Z^3 \oplus \Z/2\Z \oplus (\Z/6\Z)^2$ \\
    $A13$ & $(1,1)$ & $\Z^{36}$ \\
    $A35$ & $(1,1)$ & $\Z^{33}$ \\
        & $(-1,1)$ & $\Z^{24} \oplus (\Z/2\Z)^9$ \\
    $C24$ & $(1,1)$ & $\Z^{30}$ \\
\end{tabular}

\end{center}

\end{example}

\begin{example}
The following quasigroup satisfies 
$A35$:
   \[ A_{18} = \left(
\begin{array}{cccccc}
 0 & 1 & 2 & 3 & 4 & 5 \\
 1 & 5 & 0 & 4 & 3 & 2 \\
 2 & 0 & 4 & 5 & 1 & 3 \\
 3 & 4 & 5 & 0 & 2 & 1 \\
 4 & 2 & 3 & 1 & 5 & 0 \\
 5 & 3 & 1 & 2 & 0 & 4 \\
\end{array}
\right)\]

\bigskip

$H_1(A_{18};t,s):$

\begin{center}
    
\begin{tabular}{c|c}
    $(t,s)$ & $H_1(A_{18};t,s)$  \\
    \hline
    $(1,1)$ & $\Z/2\Z$ \\
    $(-1,1)$ & $\Z/2\Z$ \\
\end{tabular}
\end{center}
\bigskip

$H_2(A_{18};Vij,t,s):$

\begin{center}
    
\begin{tabular}{c|c|c}
    $Vij$ & $(t,s)$ & $H_2(A_{18};t,s)$  \\
    \hline

    $A35$ & $(1,1)$ & $\Z^5$ \\
        & $(-1,1)$ & $\Z^2 \oplus (\Z/2\Z)^3$ \\
    \end{tabular}

\end{center}

\end{example}

\begin{example}
The following quasigroup satisfies 
$C24$:
   \[ A_{19} = \left(
\begin{array}{cccccc}
 0 & 1 & 2 & 3 & 4 & 5 \\
 1 & 2 & 3 & 0 & 5 & 4 \\
 2 & 4 & 5 & 1 & 3 & 0 \\
 3 & 5 & 4 & 2 & 0 & 1 \\
 4 & 0 & 1 & 5 & 2 & 3 \\
 5 & 3 & 0 & 4 & 1 & 2 \\
\end{array}
\right)\]

\bigskip

$H_1(A_{19};t,s):$

\begin{center}
    
\begin{tabular}{c|c}
    $(t,s)$ & $H_1(A_{19};t,s)$  \\
    \hline
    $(1,1)$ & $\Z/2\Z$ \\
\end{tabular}
\end{center}
\bigskip

$H_2(A_3;Vij,t,s):$

\begin{center}
    
\begin{tabular}{c|c|c}
    $Vij$ & $(t,s)$ & $H_2(A_{19};t,s)$  \\
    \hline
    $C24$ & $(1,1)$ & $\Z^4$ \\
\end{tabular}

\end{center}

\end{example}

\section{Summary and speculations}\label{S:summary}

Our work is the first step in building the homology of quasigroups of Bol-Moufang type. We obtain concrete boundary maps $\partial_2$ and $\partial_3$ based on analysis of extensions, which gives a definition of first and second homology (Section~\ref{S:EQR}). We show that affine parameters $t$ and $s$ always commute (Section~\ref{sec:substitutions}) and calculate $H_1$ and $H_2$ for the examples given in \cite{PhVo1} and \cite{PhVo2} (Section~\ref{S:examples}). 

Since $\partial_2(x,y) = tx + sy - xy$ depends only on $t$ and $s$, we see that $H_1(Q;t,s)$ also depends only on $t$ and $s$, and we are able to interpret it in some cases as being an `abelianization' of the quasigroup $Q$ or one of its parastrophes (Section~\ref{S:H1Ab}).

For second homology, there is a dependence on both the choice of substitution $(t,s)$ and on the Bol-Moufang type variety one views the quasigroup as belonging to. We check in our examples that there is no dependence on the particular identity $Vij$ used to define the variety, but have not proven this in general. Thus, we have the following conjecture:

\begin{conjecture}\label{conj:identity-independence}
    Our Bol-Moufang homology depends only on the variety the quasigroup belongs to, along with the choice of $(t,s)$ substitution. That is, if $Vij$ and $Wk\ell$ define the same variety of quasigroups, then for any quasigroup $X$ in that variety, $H_2(X;Vij,t,s) = H_2(X;Wk\ell,t,s)$ for all $(t,s)$. Moreover, when $X$ is a group, any identity $Vij$ defining the variety of groups gives the usual group homology of $X$. 
\end{conjecture}

If the conjecture is true, then we may speak of, for example, {\it the} RG1 homology of a quasigroup, without specifying whether the homology is computed using identity A25 or D25. 

There is also the question of how our Bol-Moufang homology compares to more general theories. In the following subsection, we describe how one can approach general homology using a definition available in any small category initiated by Watts \cite{Wat} (see also \cite{Lod, PrWan}). This approach leads to a standard definition of group homology, see \cite{Lod}. 
\color{black}

\subsection{Homology of a small category with coefficients in a functor to $k$-modules}

Very generally, we have the following definition of homology of a small category $\mathcal{P}$ (that is, a category in which the collection of objects forms a set). Often we consider categories $\mathcal{P}$ with only one object, for instance in the case of group homology, in which we consider the category with a single object and a morphism for each element of the group, where morphisms compose according to the group operation. 

\begin{definition}\label{NerveHomology}
Let ${\mathcal P}$ be as small category (i.e. objects, $P=Ob({\mathcal P)}$ form a set), and let ${\mathcal F}:{\mathcal P} \to R$-Mod
be functor from ${\mathcal P}$ to the category of modules over a commutative ring $R$.
We call the sequence of objects and functors, $x_0 \stackrel{f_0}{\to} x_1 \stackrel{f_1}{\to} \ldots
\stackrel{f_{n-1}}{\to} x_n$ an $n$-chain  (more formally $n$-chain in the nerve of the category).
We define the chain complex $C_*({\mathcal P},{\mathcal F})$ as follows:
$$ C_n= \bigoplus_{x_0 \stackrel{f_0}{\to} x_1 \stackrel{f_1}{\to} \ldots \stackrel{f_{n-1}}{\to} x_n} {\mathcal F}(x_0)$$
where the sum is taken over all $n$-chains.

The boundary operation $\partial_n: C_n({\mathcal P},{\mathcal F}) \to C_{n-1}({\mathcal P},{\mathcal F})$ is given by:
$$\partial_n(\lambda;x_0\stackrel{f_0}{\to} x_1 \stackrel{f_1}{\to} \ldots \stackrel{f_{n-1}}{\to} x_n)=
({\mathcal F}(x_0\stackrel{f_0}{\to} x_1)(\lambda);  x_1 \stackrel{f_1}{\to} \ldots \stackrel{f_{n-1}}{\to} x_n)+$$
$$\sum_{i=1}^n (-1)^i(\lambda;x_0\stackrel{f_0}{\to} x_1 \stackrel{f_1}{\to}\ldots \to
x_{i-1}\stackrel{f_if_{i-1}}{\to} x_{i+1}\to \ldots \stackrel{f_{n-1}}{\to} x_n).$$
We denote by $H_n({\mathcal P},{\mathcal F})$ the homology yielded by the above chain complex.
\end{definition}

We should stress that the above chain complex has the structure of a simplicial module. That is $\partial_n=\sum_{i=0}^n(-1)^id_i$, 
where $$d_0(\lambda;x_0\stackrel{f_0}{\to} x_1 \stackrel{f_1}{\to} \ldots \stackrel{f_{n-1}}{\to} x_n)=
({\mathcal F}(x_0\stackrel{f_0}{\to} x_1)(\lambda);  x_1 \stackrel{f_1}{\to} \ldots \stackrel{f_{n-1}}{\to} x_n).$$
and for $i>0$ we define
$$d_i(\lambda;x_0\stackrel{f_0}{\to} x_1 \stackrel{f_1}{\to} \ldots \stackrel{f_{n-1}}{\to} x_n)=(\lambda;x_0\stackrel{f_0}{\to} x_1 \stackrel{f_1}{\to}\ldots \to
x_{i-1}\stackrel{f_if_{i-1}}{\to} x_{i+1} \ldots \stackrel{f_{n-1}}{\to} x_n).$$ 
Furthermore, the degeneracy map $s_i=s_{i,n}: C_n \to C_{n+1}$ is given by inserting the identity map, that is $$s_i(\lambda;x_0\stackrel{f_0}{\to} x_1 \stackrel{f_1}{\to} \ldots \stackrel{f_{n-1}}{\to} x_n)=$$

$$(\lambda;x_0\stackrel{f_0}{\to} x_1 \stackrel{f_1}{\to} \ldots \stackrel{f_{i-1}}{\to} x_i \stackrel{Id}{\to} x_i \stackrel{f_i}{\to} x_{i+1} \stackrel{f_{i+1}}{\to}  \ldots \stackrel{f_{n-1}}{\to} x_n).$$

For completeness we recall that for any presimplicial module we have naturally defined chain complexes and homology:
\begin{definition} For a presimplicial module ${\mathcal M}$, that is a collection of modules $M_n$, $n\geq 0$, together with maps,
called  maps or face operators,
$$d_i: M_n \to  M_{n-1}, \ \ 0\leq i \leq n, $$
such that:
$$d_id_j = d_{j-1}d_i,\ \ 0\leq i < j \leq n, $$
we define a chain complex with chain groups $M_n$ and a boundary map $\partial_n: M_n \to M_{n-1}$ given by:
$$\partial_n=\sum_{i=0}^n (-1)^i d_i$$
One easily checks that $\partial^2=0$ and thus $({\mathcal M},\partial)$ is a chain complex\footnote{In fact,
$\partial^2= (\sum_{i=0}^{n-1} (-1)^i d_i)(\sum_{i=0}^n (-1)^i d_i)=
 \sum_{0\leq i<j\leq n} ((-1)^{i+j}d_id_j + (-1)^{i+j-1}d_{j-1}d_i)=0$.}
and homology $H_*({\mathcal M})$ can be defined from this chain complex.
\end{definition}

The homology definition here also requires a functor $\mathcal{F}$ to a category of modules. Frequently we take the constant functor $|\Z|$ to the category of abelian groups, for example in the group case, $\mathcal{F}$ takes the single object $\bullet$ of $\mathcal{P}$ to $\Z$ and all morphisms to $1_\Z$. This means the $n$-chains in this case are of the form \[ \begin{tikzcd}
    \bullet \arrow[r, "g_0"] & \bullet \arrow[r, "g_1"] & \cdots \arrow[r, "g_{n - 1}"] & \bullet,
\end{tikzcd} \] so $n$-tuples $(g_0, \dots, g_{n - 1})$ of elements of the group, and the face maps $d_i$ are exactly the usual $d_i(g_0, \dots, g_{n - 1}) = (g_0, \dots, g_{i - 1} g_i, \dots, g_{n - 1})$. 

In setting up an analogous construction for a quasigroup $X$, we can set up a small category with one object, and we have two natural choices for what the morphisms ought to be. First, we can take our one object to be the quasigroup $X$ and the morphisms to be all quasigroup endomorphisms of $X$. 

For example, when $X$ is the quasigroup $A_1$ with operation table \[ A_1 = \left(
\begin{array}{cccc}
 0 & 1 & 2 & 3 \\
 2 & 0 & 3 & 1 \\
 1 & 3 & 0 & 2 \\
 3 & 2 & 1 & 0 \\
\end{array}
\right) \] from Example 5.1, any endomorphism must have $\varphi(0) = 0$, and there are four of these, determined by $\varphi(1)$, so we denote the endomorphism with $1 \mapsto i$ by $\varphi_i$. One checks the morphisms then compose according to \[ \left(
\begin{array}{cccc}
 0 & 0 & 0 & 0 \\
 0 & 1 & 2 & 3 \\
 0 & 2 & 1 & 3 \\
 0 & 3 & 3 & 0 \\
\end{array}
\right) \] so for instance, $\varphi_2 \circ \varphi_3 = \varphi_3$. Thus the small category has one object and four morphisms, setting up the chain complex \[ \begin{tikzcd}
    & \cdots \arrow[r, "\partial_4"] & \Z^{64} \arrow[r, "\partial_3"] & \Z^{16} \arrow[r, "\partial_2"] & \Z^4 \arrow[r, "\partial_1"] & \Z \arrow[r] & 0
\end{tikzcd} \] where we can again think of an $n$-chain as a linear combination of $n$-tuples, this time with the tuple entries $\varphi_i$. Then \begin{align*}
    \partial_1(\varphi_i) &= \text{unique 0-chain} - \text{unique 0-chain} = 0 \\
    \partial_2(\varphi_i, \varphi_j) &= \varphi_j - \varphi_j \circ \varphi_i + \varphi_i \\
    \partial_3(\varphi_i, \varphi_j, \varphi_k) &= (\varphi_j, \varphi_k) - (\varphi_j \circ \varphi_i, \varphi_k) + (\varphi_i, \varphi_k \circ \varphi_j) - (\varphi_i, \varphi_j)
\end{align*} and so on. Thus one computes in the case of $A_1$ that \[ \partial_2 = \begin{pmatrix}
    1 & 0 & 0 & 0 & 0 & 0 & 0 & 0 & 0 & 0 & 0 & 0 & 0 & 0 & 0 & 0 \\
    0 & 1 & 0 & 0 & 1 & 1 & 1 & 1 & 0 & 1 & -1 & 0 & 0 & 1 & 0 & 0 \\
    0 & 0 & 1 & 0 & 0 & 0 & 0 & 0 & 1 & 0 & 2 & 1 & 0 & 0 & 1 & 0 \\
    0 & 0 & 0 & 1 & 0 & 0 & 0 & 0 & 0 & 0 & 0 & 0 & 1 & 0 & 0 & 2
\end{pmatrix} \] so that $H_1$ is trivial. One similarly computes that $H_2$ is also trivial. We note that this does not match any of our Bol-Moufang homology computations for $A_1$, as there we always have a nontrivial $H_1$ in every substitution we consider. 

A second approach uses the notion of the multiplication group of a quasigroup (see the introduction to Section~\ref{S:QBM}). In this case the small category again has a single object $X$, where the morphisms now are the permutations of $X$ as a set coming from $\Mlt(X)$. Composition in the category is then the group operation of $\Mlt(X)$, and we see from the discussion above that this means the homology of this small category is the group homology of $\Mlt(X)$. Again considering the example $A_1$, we have the dihedral group of order 8 as $\Mlt(A_1)$, and so the homology we obtain from this approach is $H_1(D_4) \cong (\Z/2\Z)^2$ and $H_2(D_4) \cong \Z/2\Z$, which again does not match any of our Bol-Moufang homology calculations for $A_1$ in Section~\ref{S:examples}. 

\vspace{1em}

Thus we conclude by inviting the reader to take further steps in building homology for Bol-Moufang quasigroups, perhaps extending the constructions here to define higher boundary maps $\partial_n$, relating them in some way to a more general (co)homology definition (e.g. using the small category approach outlined above, or for a different approach using monads, see \cite{Dus} and \cite{Smi1}), or finding applications, especially in knot theory.

\section{Acknowledgments}
The fourth author was supported by the LAMP Program of the National Research Foundation of Korea, grant No. RS-2023-00301914, and by the MSIT grant No. 2022R1A5A1033624. The fifth author was partially supported by the Simons Collaboration Grant 637794. The last author is grateful for invitation to George Washington University and hospitality of J\'ozef Przytycki and his wife, Teresa.

\appendix

\section{Tables for $H(T_i)-H(T_j)$ and $ Q(T_i) - Q(T_j)$} 

Recall that each Bol-Moufang identity $Xij$ on four (not necessarily distinct) letters are realized by two binary trees, $T_1$ and $T_2$.


List \ref{list} collects polynomials $H(T_i)- H(T_j)$. We recall that 
$\partial_2 \circ \partial_3 = \partial_2(Q(T_1)) - \partial_2(Q(T_2))\stackrel{x_i\mapsto a_i}{=} H(T_1)-H(T_2)$ for some $T_1, T_2$ with leaves decorated by $x_i$. Table 3 shows $H(T_i) - H(T_j)$.

\begin{table}\label{NewHTable}
    \makebox[\textwidth][c]{
      \resizebox{1.2\textwidth}{!}{
       \renewcommand{\arraystretch}{2.5}
          \setlength{\extrarowheight}{2pt}
        \begin{tabular}{|c|c|c|c|c|c|}
        \hline
        &$T_1$ & $T_2$ & $T_3$ & $T_4$ & $T_5$ \\
        \hline
    $T_1$ & 0 & \makecell{ \hspace{2pt} \\ $ st(1-t)a_2 +  s(st-ts)a_3 $ \\ $ +s^2(s-1) a_4$ \\ \hspace{2pt} } & \makecell{ \hspace{2pt} \\ $t(1-t)a_1 + (st-ts)a_2 $ \\ $ +(s-1)st a_3 + s^2(s-1)a_4 $ \\ \hspace{2pt} } & \makecell{ \hspace{2pt} \\ $ t(1-t)a_1 + (1 - t)st a_2  $ \\ $ +(s^2t-ts^2)a_3 + s(s^2 - 1)a_4$ \\ \hspace{2pt} } &  \makecell{ \hspace{2pt} \\ $t(1-t^2)a_1 + (st - t^2s) a_2  $ \\ $+ (s^2t - ts)a_3 + s(s^2 - 1)a_4 $ \\ \hspace{2pt} } \\

    \hline
    $T_2$ & & 0 & \makecell{ \hspace{2pt} \\ $t(1-t)a_1 + (st^2 - ts)a_2 $ \\ $ + st(s - 1)a_3 $ \\ \hspace{2pt} } & \makecell{ \hspace{2pt} \\ $t(1 - t) a_1 + (st^2-tst)a_2  $ \\ $ + (sts-ts^2)a_3 + s(s - 1)a_4$ \\ \hspace{2pt} } & \makecell{ \hspace{2pt} \\ $t(1-t^2)a_1 + (st^2 -t^2 s)a_2  $ \\ $ + (s - 1)ts a_3 +s(s - 1) a_4$ \\ \hspace{2pt} } \\
    \hline
    $T_3$ & & & 0  & \makecell{ \hspace{2pt} \\ $ts(1 - t)a_2 +(st-ts^2)a_3  $ \\ $ +s(s-1)a_4 $ \\ \hspace{2pt} } & \makecell{ \hspace{2pt} \\ $t^2(1 - t)a_1 + t(1 - t) s a_2 $ \\ $ +(st - ts)a_3 + s(s - 1)a_4$ \\ \hspace{2pt} } \\
    \hline
    $T_4$ & & & & 0 & \makecell{ \hspace{2pt} \\ $t^2(1 - t)a_1 + t(st - t s)a_2 $ \\ $ + ts(s - 1)a_3 $ \\ \hspace{2pt} } \\
    \hline
    $T_5$ & & & & & 0 \\
    \hline
        \end{tabular}
      }
    }

    \caption{$H(\text{row tree}) - H(\text{column tree})$ for the five parenthesization trees with leaves labelled by $a_1, a_2, a_3, a_4$. Blank entry $a_{ij}$ is $-a_{ji}$; i.e. the full matrix is skew symmetric.}
    \end{table}

  Table 4 shows $Q(T_1) - Q(T_2)$. We remark that $\partial^{Xij}_3(x,y,z) = Q(T_1) - Q(T_2)$. The list of $\partial_3$ maps is given in Section \ref{S:homologyboundary}.

\begin{table}\label{NewHTable}
    \makebox[\textwidth][c]{
      \resizebox{1.2\textwidth}{!}{
       \renewcommand{\arraystretch}{2.5}
          \setlength{\extrarowheight}{2pt}
        \begin{tabular}{|c|c|c|c|c|c|}
        \hline
        &$T_1$ & $T_2$ & $T_3$ & $T_4$ & $T_5$ \\
        \hline
    $T_1$ & 0 & \makecell{ \hspace{2pt} \\ $(x_1, x_2(x_3 x_4))+ s(x_2, x_3 x_4)$ \\ $+ s^2 (x_3, x_4) - (x_1, (x_2 x_3)x_4) $ \\ $- s(x_2 x_3, x_4) - st(x_2, x_3) $ \\ \hspace{2pt} } & \makecell{ \hspace{2pt} \\ $(x_1, x_2(x_3 x_4))+ s(x_2, x_3 x_4) $ \\ $+ s^2 (x_3, x_4)- (x_1 x_2, x_3 x_4)  $ \\ $- t(x_1, x_2) - s(x_3, x_4)$ \\ \hspace{2pt} } & \makecell{ \hspace{2pt} \\ $(x_1, x_2(x_3 x_4))+ s(x_2, x_3 x_4)  $ \\ $ + s^2 (x_3, x_4) - (x_1(x_2 x_3), x_4)  $ \\ $ - t(x_1, x_2 x_3) - ts(x_2 , x_3)$ \\ \hspace{2pt} } &  \makecell{ \hspace{2pt} \\ $(x_1, x_2(x_3 x_4))+ s(x_2, x_3 x_4)  $ \\ $+ s^2 (x_3, x_4) - ((x_1 x_2) x_3, x_4)  $ \\ $ - t (x_1 x_2, x_3) - t^2 (x_1, x_2)$ \\ \hspace{2pt} } \\

    \hline
    $T_2$ & & 0 & \makecell{ \hspace{2pt} \\ $(x_1, (x_2 x_3)x_4) + s(x_2 x_3, x_4) $ \\ $ +st(x_2, x_3)- (x_1 x_2, x_3 x_4) $ \\ $ - t(x_1, x_2) - s(x_3, x_4)$ \\ \hspace{2pt} } & \makecell{ \hspace{2pt} \\ $(x_1, (x_2 x_3)x_4) + s(x_2 x_3, x_4) $ \\ $ +st(x_2, x_3)  - (x_1(x_2 x_3), x_4) $ \\ $ - t(x_1, x_2 x_3) - ts(x_2 , x_3)$ \\ \hspace{2pt} } & \makecell{ \hspace{2pt} \\ $(x_1, (x_2 x_3)x_4) + s(x_2 x_3, x_4) $ \\ $ +st(x_2, x_3)  -  ((x_1 x_2) x_3, x_4) $ \\ $ - t (x_1 x_2, x_3) - t^2 (x_1, x_2)$ \\ \hspace{2pt} } \\
    \hline
    $T_3$ & & & 0  & \makecell{ \hspace{2pt} \\ $(x_1 x_2, x_3 x_4) + t(x_1, x_2) $ \\ $ + s(x_3, x_4) - (x_1(x_2 x_3), x_4) $ \\ $ -t(x_1, x_2 x_3) - ts(x_2 , x_3)$ \\ \hspace{2pt} } & \makecell{ \hspace{2pt} \\ $ (x_1 x_2, x_3 x_4) + t(x_1, x_2)  $ \\ $ + s(x_3, x_4) - ((x_1 x_2) x_3, x_4) $ \\ $ - t (x_1 x_2, x_3) - t^2 (x_1, x_2) $ \\ \hspace{2pt} } \\
    \hline
    $T_4$ & & & & 0 & \makecell{ \hspace{2pt} \\ $(x_1(x_2 x_3), x_4) + t(x_1, x_2 x_3) $ \\ $ + ts(x_2 , x_3) - ((x_1 x_2) x_3, x_4) $ \\ $ - t (x_1 x_2, x_3) - t^2 (x_1, x_2)$ \\ \hspace{2pt} } \\
    \hline
    $T_5$ & & & & & 0 \\
    \hline
        \end{tabular}
      }
    }

    \caption{$Q(\text{row tree}) - Q(\text{column tree})$ for the five parenthesization trees with leaves labelled by $x_1, x_2, x_3, x_4$. Blank entry $a_{ij}$ is $-a_{ji}$; i.e. the full matrix is skew symmetric.}
    \end{table}

\end{document}